\documentclass[11pt]{article}
\usepackage[utf8]{inputenc}
\usepackage{amsmath}
\usepackage{amsfonts}
\usepackage{amssymb}
\usepackage{amsthm}
\usepackage[english]{babel}
\usepackage[utf8]{inputenc}
\usepackage{hyperref}
\usepackage[alphabetic]{amsrefs}
\usepackage[left=3cm,right=3cm,top=3cm,bottom=3cm]{geometry}
\usepackage{graphicx}
\usepackage{mathrsfs}
\hypersetup{
	colorlinks   = true, 
	urlcolor     = blue, 
	linkcolor    = blue,
	citecolor    = blue 
}
\usepackage{titlesec}
\titlespacing{\section}{0pt}{1ex}{-1ex}
\titlespacing{\subsection}{0pt}{1ex}{-1ex}

\begin{document}
\title{\textbf{A comparison of the Almgren-Pitts and the Allen-Cahn min-max theory}}

\author{Akashdeep Dey \thanks{Email: adey@math.princeton.edu, dey.akash01@gmail.com}}
\date{}
\maketitle

\newtheorem{thm}{Theorem}[section]
\newtheorem{lem}[thm]{Lemma}
\newtheorem{pro}[thm]{Proposition}
\newtheorem{clm}[thm]{Claim}
\newtheorem*{thm*}{Theorem}
\newtheorem*{lem*}{Lemma}
\newtheorem*{clm*}{Claim}

\theoremstyle{definition}
\newtheorem{defn}[thm]{Definition}
\newtheorem{ex}[thm]{Example}

\theoremstyle{remark}
\newtheorem{rmk}[thm]{Remark}

\numberwithin{equation}{section}

\newcommand{\mf}{manifold\;}
\newcommand{\vf}{varifold\;}
\newcommand{\hy}{hypersurface\;}
\newcommand{\Rm}{Riemannian\;}
\newcommand{\cn}{constant\;}
\newcommand{\mt}{metric\;} 
\newcommand{\st}{such that\;}
\newcommand{\Thm}{Theorem\;}
\newcommand{\Lem}{Lemma\;}
\newcommand{\Pro}{Proposition\;}
\newcommand{\Eqn}{Equation\;}
\newcommand{\eq}{equation}
\newcommand{\te}{there exists\;}
\newcommand{\tf}{Therefore, \;}
\newcommand{\wrt}{with respect to\;}
\newcommand{\bbr}{\mathbb{R}}
\newcommand{\bbn}{\mathbb{N}}
\newcommand{\bbz}{\mathbb{Z}}
\newcommand{\mres}{\mathbin{\vrule height 1.6ex depth 0pt width 0.13ex\vrule height 0.13ex depth 0pt width 1.3ex}}
\newcommand{\ra}{\rightarrow}
\newcommand{\lra}{\longrightarrow}
\newcommand{\fn}{function\;}
\newcommand{\sps}{Suppose\;}
\newcommand{\del}{\partial}
\newcommand{\seq}{sequence\;}
\newcommand{\cts}{continuous\;} 
\newcommand{\bF}{\mathbf{F}} 
\newcommand{\bM}{\mathbf{M}} 
\newcommand{\bL}{\mathbf{L}}
\newcommand{\cm}{\mathcal{C}(M)}
\newcommand{\zn}{\mathcal{Z}_n(M^{n+1}; \mathbb{Z}_2)}
\newcommand{\bbb}{\mathbf{B}}
\newcommand{\tx}{\tilde{X}}
\newcommand{\xt}{\tilde{X}}
\newcommand{\xtk}{\tilde{X}[K]}
\newcommand{\tph}{\tilde{\Phi}}
\newcommand{\pht}{\tilde{\Phi}}
\newcommand{\delst}{\partial^{*}}
\newcommand{\pst}{\tilde{\Psi}}

\newcommand{\cA}{\mathcal{A}}\newcommand{\cB}{\mathcal{B}}\newcommand{\cC}{\mathcal{C}}\newcommand{\cD}{\mathcal{D}}\newcommand{\cE}{\mathcal{E}}\newcommand{\cF}{\mathcal{F}}\newcommand{\cG}{\mathcal{G}}\newcommand{\cH}{\mathcal{H}}\newcommand{\cI}{\mathcal{I}}\newcommand{\cJ}{\mathcal{J}}\newcommand{\cK}{\mathcal{K}}\newcommand{\cL}{\mathcal{L}}\newcommand{\cM}{\mathcal{M}}\newcommand{\cN}{\mathcal{N}}\newcommand{\cO}{\mathcal{O}}\newcommand{\cP}{\mathcal{P}}\newcommand{\cQ}{\mathcal{Q}}\newcommand{\cR}{\mathcal{R}}\newcommand{\cS}{\mathcal{S}}\newcommand{\cT}{\mathcal{T}}\newcommand{\cU}{\mathcal{U}}\newcommand{\cV}{\mathcal{V}}\newcommand{\cW}{\mathcal{W}}\newcommand{\cX}{\mathcal{X}}\newcommand{\cY}{\mathcal{Y}}\newcommand{\cZ}{\mathcal{Z}}

\newcommand{\al}{\alpha}\newcommand{\be}{\beta}\newcommand{\ga}{\gamma}\newcommand{\de}{\delta}\newcommand{\ve}{\varepsilon}\newcommand{\et}{\eta}\newcommand{\ph}{\phi}\newcommand{\vp}{\varphi}\newcommand{\ps}{\psi}\newcommand{\ka}{\kappa}\newcommand{\la}{\lambda}\newcommand{\om}{\omega}\newcommand{\rh}{\rho}\newcommand{\si}{\sigma}\newcommand{\tht}{\theta}\newcommand{\ta}{\tau}\newcommand{\ch}{\chi}\newcommand{\ze}{\zeta}\newcommand{\Ga}{\Gamma}\newcommand{\De}{\Delta}\newcommand{\Ph}{\Phi}\newcommand{\Ps}{\Psi}\newcommand{\La}{\Lambda}\newcommand{\Om}{\Omega}\newcommand{\Si}{\Sigma}\newcommand{\Th}{\Theta}

\newcommand{\norm}[1]{\left\|#1\right\|}\newcommand{\md}[1]{\left|#1\right|}\newcommand{\Md}[1]{\Big|#1\Big|}
\newcommand{\db}[1]{[\![#1]\!]}
\newcommand{\ov}[1]{\overline{#1}}
\newcommand{\vol}{\operatorname{Vol}}
\newcommand{\spt}{\operatorname{spt}}\newcommand{\Ind}{\operatorname{Ind}}
\newcommand{\vto}{\vspace{-2ex}}\newcommand{\vth}{\vspace{-3ex}}\newcommand{\vfo}{\vspace{-4ex}}

\vspace{-4ex}
\begin{abstract}
	\vspace{-1.5ex}
	\noindent
	Min-max theory for the Allen-Cahn equation was developed by Guaraco \cite{Guaraco} and Gaspar-Guaraco \cite{GG1}. They showed that the Allen-Cahn widths are greater than or equal to the Almgren-Pitts widths. In this article we will prove that the reverse inequalities also hold i.e. the Allen-Cahn widths are less than or equal to the Almgren-Pitts widths. Hence, the Almgren-Pitts widths and the Allen-Cahn widths coincide. We will also show that all the closed minimal hypersurfaces (with optimal regularity) which are obtained from the Allen-Cahn min-max theory are also produced by the Almgren-Pitts min-max theory. As a consequence, we will point out that the index upper bound in the Almgren-Pitts setting, proved by Marques-Neves \cite{MN_index} and Li \cite{Li_index}, can also be obtained from the index upper bound in the Allen-Cahn setting, proved by Gaspar \cite{Gaspar} and Hiesmayr \cite{H}.
\end{abstract}

\section{Introduction}
Minimal submanifolds are defined by the condition that they are the critical points of the area functional. In \cite{alm_article,Alm}, Almgren studied the topology of the space of cycles and developed a min-max theory for the area functional. He proved that any closed, Riemannian manifold \((M^{n+1},g)\) contains a minimal variety of dimension \(l\) for every \(1\leq l \leq n\). The regularity theory in the co-dimension $1$ case was further developed by Pitts \cite{Pitts} and Schoen-Simon \cite{SS}. They proved that in a closed, Riemannian manifold \((M^{n+1},g)\), \(n+1 \geq 3\), there exists a closed, minimal hypersurface which is smooth and embedded outside a singular set of Hausdorff dimension \(\leq n-7.\)

In recent years, there have been a lot of research activities in the Almgren-Pitts min-max theory. By the work of Marques-Neves \cite{MN_ricci_positive} and Song \cite{Song}, every closed Riemannian manifold \((M^{n+1},g)\), \(3\leq n+1\leq 7\), contains infinitely many closed, minimal hypersurfaces. This was conjectured by Yau \cite{yau}. In \cite{IMN}, Irie, Marques and Neves proved that for a generic \mt \(g\) on \(M\), the union of all closed, minimal hypersurfaces is dense in \((M,g)\). This theorem was later quantified by Marques, Neves and Song in \cite{MNS} where they proved that for a generic metric there exists an equidistributed sequence of closed, minimal hypersurfaces in \((M,g)\). In higher dimensions, Li \cite{Li} has proved that every closed Riemannian manifold equipped with a generic \mt contains infinitely many closed minimal hypersurfaces with optimal regularity. The Weyl law for the volume spectrum $\{\om_k\}_{k=1}^{\infty}$, proved by Liokumovich, Marques and Neves \cite{LMN} played a major role in the arguments of \cite{IMN,MNS,Li}.

The Morse index of the minimal hypersurfaces produced by the Almgren-Pitts min-max theory has been obtained by Marques and Neves when the ambient dimension $3\leq n+1\leq 7$. In \cite{MN_index}, Marques and Neves showed that the index of the min-max minimal hypersurface is bounded from above by the dimension of the parameter space. Zhou \cite{Zhou2} has proved that for a generic (bumpy) metric, the min-max minimal hypersurfaces have multiplicity one which was conjectured by Marques and Neves. Using the Morse index upper bound \cite{MN_index} and multiplicity one theorem \cite{Zhou2}, Marques and Neves \cite{MN_index_2} have proved the following theorem. For a generic (bumpy) \mt there exists a sequence of closed, embedded, two-sided minimal hypersurfaces \(\{\Si_k\}_{k=1}^{\infty}\) in \((M^{n+1},g)\) \st \(\Ind(\Si_k)=k\) and \(\cH^n(\Si_k)=\om_k\sim k^{\frac{1}{n+1}}\). This theorem has been generalized by Marques, Montezuma and Neves in \cite{mmn} where they have proved the strong Morse inequalities for the area functional. In higher dimensions, Morse index upper bound has been proved by Li \cite{Li_index}.

In \cite{Guaraco}, Guaraco introduced a new approach for the min-max construction of minimal hypersurfaces which was further developed by Gaspar and Guaraco in \cite{GG1}. This approach is based on the study of the limiting behaviour of solutions to the Allen-Cahn equation. The Allen-Cahn equation (with parameter \(\ve >0\)) is the following semi-linear, elliptic PDE
\[AC_{\ve}(u):=\ve\De u-\ve^{-1} W'(u)=0\]
where \(W:\bbr \ra \bbr\) is a double well potential e.g. \(W(t)=\frac{1}{4}(1-t^2)^2.\) The solutions of this equation are precisely the critical points of the energy functional
\[E_{\ve}(u)=\int_M\ve\frac{|\nabla u|^2}{2}+\frac{W(u)}{\ve}.\]
Building on the work of Hutchinson-Tonegawa \cite{HT}, Tonegawa \cite{t} and Tonegawa-Wickramasekera \cite{TW}, Guaraco \cite{Guaraco} proved that if \(\{u_i\}_{i=1}^{\infty}\) is a sequence of solutions to the Allen-Cahn equation \(AC_{\ve_i}(u_i)=0\), \(\ve_i \ra 0\) with \(E_{\ve_i}(u_i)\) and \(\Ind(u_i)\) are uniformly bounded, then, possibly after passing to a subsequence, the level sets of \(u_i\) accumulate around a closed, minimal hypersurface with optimal regularity. (Such a minimal hypersurface is called a \textit{limit-interface}.) Moreover, by a mountain-pass argument he proved the existence of critical points of \(E_{\ve}\) with uniformly bounded energy and Morse index. In this way he obtained a new proof of the previously mentioned theorem of Almgren-Pitts-Schoen-Simon. The index of the limit-interface is bounded by the index of the solutions. This was proved by Hiesmayr \cite{H} assuming the limit-interface is two-sided and by Gaspar \cite{Gaspar} in the general case. 

In \cite{GG1,GG}, Gaspar and Guaraco studied the phase transition spectrum which is the Allen-Cahn analogue of the volume spectrum. They proved that the phase transition spectrum satisfies a Weyl law similar to the volume spectrum and gave alternative proofs of the density \cite{IMN} and the equidistribution \cite{MNS} theorems. In \cite{CM}, Chodosh and Mantoulidis proved the multiplicity one conjecture in the Allen-Cahn setting in dimension \(3\) and the upper semicontinuity of the Morse index when the limit-interface has multiplicity one. As a consequence, they proved that for a generic (bumpy) metric \(g\) on a closed manifold \(M^3\), there exists a sequence of closed, embedded, two-sided minimal surfaces \(\{\Si_k\}_{k=1}^{\infty}\) in \((M^3,g)\) \st \(\Ind(\Si_k)=k\) and \(\text{area}(\Si_k)\sim k^{1/3}\). 

If \(\Si\) is a non-degenerate, separating, closed, embedded minimal hypersurface in a closed Riemannian manifold, Pacard and Ritor\'{e} \cite{pr} constructed solutions of the Allen-Cahn equation \(AC_{\ve}(u)=0\) for sufficiently small \(\ve>0\) whose level sets converge to \(\Si\). The uniqueness of these solutions has been proved by Guaraco, Marques and Neves \cite{gmn}. The construction of Pacard and Ritor\'{e} has been extended by Caju and Gaspar \cite{cg} in the case when all the Jacobi fields of \(\Si\) are induced by the ambient isometries.

In the present article we will be interested in the question to what extent the Almgren-Pitts min-max theory and the Allen-Cahn min-max theory agree. Part of this question has been answered by Guaraco \cite{Guaraco} and Gaspar-Guaraco \cite{GG1}; they proved that the Almgren-Pitts widths are less than or equal to the Allen-Cahn widths. The aim of this article is to prove the reverse inequality i.e the Allen-Cahn widths are less than or equal to the Almgren-Pitts widths.

To precisely state our main result, we need some facts about the universal \(G\)-principal bundle. We will follow the book by Dieck \cite{Dieck}*{Chapter 14.4} and the paper by Gaspar and Guaraco \cite{GG1}*{Appendix B} where further details can be found. Let \(G\) be a topological group and \(p_{G}:EG \ra BG\) be a universal \(G\)-principal bundle (which is unique upto isomorphism). Given a topological space \(B\), there exists a one-to-one correspondence between the set of homotopy classes of maps \(B \ra BG\) and the set of isomorphism classes of numerable \(G\)-principal bundles over \(B\). If \(f_1,f_2:B \ra BG\) are homopotic, \(f_1^{*}EG\) and \(f_2^{*}EG\) are isomorphic numerable principal \(G\)-bundles over \(B\). Conversely, if \(E\) is a numerable free \(G\)-space, there exists a \(G\)-map from \(E\) to \(EG\) which is unique upto \(G\)-homotopy. Denoting \(B=E/G\), if \(F_1,F_2:E\ra EG\) are \(G\)-maps, they descend to homotopic maps \(f_1,\;f_2 :B \ra BG\). We also note the following facts: a numerable \(G\)-principal bundle \(\mathfrak{p}:\cE \ra \cB\) is universal if \(\cE\) is a contractible topological space \cite{Dieck}*{14.4.12}; each open covering of a paracompact space is numerable \cite{Dieck}*{13.1.3}.

We refer to Section 2 for the definitions and notations used in the rest of this section. Let \((M^{n+1},g)\) be a closed Riemannian manifold, \(n+1\geq 3\). Let \(X\) be a cubical complex and we fix a double cover \(\pi:\tx \ra X\). Since the space \(\mathbf{I}_{n+1}(M^{n+1};\bF;\bbz_2)\) is contractible \cite{MN_index_2} and every metric space is paracompact, \(\del:\mathbf{I}_{n+1}(M^{n+1};\bF;\bbz_2) \ra \cZ_n(M^{n+1};\bF;\bbz_2)\) is a universal \(\bbz_2\)-principal bundle. We denote by \(\Pi\) the homotopty class of maps \(X \ra \cZ_n(M^{n+1};\bF;\bbz_2)\) corresponding to the double cover \(\pi:\tx \ra X\). More concretely, \(\Pi\) is the set of all maps \(\Ph: X \ra \cZ_n(M^{n+1};\bF;\bbz_2)\) \st \(\ker(\Ph_{*})=\text{im}(\pi_{*})\) where 
\[\Ph_{*}:\pi_1(X)\ra \pi_1\left(\cZ_n(M^{n+1};\bF;\bbz_2)\right)(=\bbz_2),\quad \pi_{*}:\pi_1(\tx)\ra \pi_1(X)\]
are the maps induced by \(\Ph,\;\pi\).

Similarly, \(H^1(M)\setminus\{0\}\) is contractible and there is a free \(\bbz_2\) action on this space given by \(u \mapsto -u\). Therefore, \(H^1(M)\setminus\{0\}\) (equipped with the \(\bbz_2\) action) is the total space of a universal \(\bbz_2\)-principal bundle. Let \(\tilde{\Pi}\) denote the set of all \(\bbz_2\)-equivariant maps \(h:\tx \ra H^1(M)\setminus\{0\}\) i.e. if \(T:\tx \ra \tx\) is the deck transformation, \(h(T(x))=-h(x)\) for all \(x \in \tx\).

The following theorem follows from the work of Guaraco \cite{Guaraco} and Gaspar-Guaraco \cite{GG1}.
\begin{thm}[\cite{Guaraco,GG1}]\label{thm GG}
Let \(\bL_{AP}(\Pi)\) be the Almgren-Pitts width of \(\Pi\) (\eqref{2 def AP width}) and \(\bL_{\ve}(\tilde{\Pi})\) be the \(\ve\)-Allen-Cahn width of \(\tilde{\Pi}\) (\eqref{2 def AC width}). Then the following inequality holds.
\begin{equation}
	\bL_{AP}(\Pi)\leq \frac{1}{2\si}\liminf_{\ve\ra 0^{+}} \bL_{\ve}(\tilde{\Pi}).\label{width ineq of GG}
\end{equation}
As a consequence, the following inequality holds between the volume spectrum and the phase transition spectrum.
\begin{equation}
\om_p \leq \frac{1}{2\si}\liminf_{\ve\ra 0^{+}}c_{\ve}(p) \; \forall p\in \bbn.\label{spec ineq of GG}
\end{equation}
\end{thm}\vspace{-2ex}
In the present article we will show that the reverse inequality also holds. More precisely, we will prove the following theorem.
\begin{thm}\label{thm main thm}
	We have the following inequality between the Almgren-Pitts width and the \(\ve\)-Allen-Cahn width.
\begin{equation}
	\frac{1}{2\si}\limsup_{\ve \ra 0^{+}} \bL_{\ve}(\tilde{\Pi}) \leq \bL_{AP}(\Pi).\label{width ineq}
\end{equation}
As a consequence we have,
\begin{equation}
	\frac{1}{2\si}\limsup_{\ve \ra 0^{+}}c_{\ve}(p)\leq \om_p\;\forall p\in \bbn.\label{spec ineq}
\end{equation}
\end{thm}\vspace{-4ex}
Hence, combining \eqref{width ineq of GG} and \eqref{width ineq} we conclude that \(\frac{1}{2\si}\lim_{\ve\ra 0^{+}}\bL_{\ve}(\tilde{\Pi})\) exists and is equal to \(\bL_{AP}(\Pi)\). Similarly, \eqref{spec ineq of GG} and \eqref{spec ineq} together imply that \(\frac{1}{2\si}\lim_{\ve\ra 0^{+}}c_{\ve}(p)\) exists and is equal to \(\om_p\) for all \(p \in \bbn\). When the ambient dimension \(3\leq n+1 \leq 7\), it was proved by Gaspar and Guaraco \cite{GG} that \(\lim_{\ve\ra 0^{+}}c_{\ve}(p)\) exists. 

The next theorem essentially follows from the work of Hutchinson-Tonegawa \cite{HT}, Guaraco \cite{Guaraco} and Gaspar-Guaraco \cite{GG1}. Informally speaking, it says that all the minimal hypersurfaces obtained from the Allen-Cahn min-max theory are also produced by the Almgren-Pitts min-max theory.

\begin{thm}\label{thm critical set}
Let \(\mathbf{C}_{AC}(\tilde{\Pi})\) be as defined at the end of Section \ref{section 2.5} and \(\mathbf{C}_{AP}(\Pi)\) be as defined at the end of Section \ref{section 2.3}. If \(V \in \mathbf{C}_{AC}(\tilde{\Pi})\), then \(V \in \mathbf{C}_{AP}(\Pi)\) as well.
\end{thm}\vspace{-2ex}
Combining the index estimate of Gaspar \cite{Gaspar} (Theorem \ref{thm index}) and the above Theorem \ref{thm critical set}, one can obtain an alternative proof of the following Morse index upper bound in the Almgren-Pitts min-max theory proved by Marques-Neves \cite{MN_index} and Li \cite{Li_index}.

\begin{thm}[\cite{MN_index,Li_index}]
	Let $\dim(X)=\dim(\tx)=k$. There exists \(V \in \mathbf{C}_{AP}(\Pi)\) \st \(\Ind(\spt(V))\) is less than or equal to \(k\).
\end{thm}\vto
Indeed, by the min-max theory for the Allen-Cahn functional (see Section 2.4), for all sufficiently small \(\ve>0\) there exists a min-max critical point \(u_{\ve}\) of \(E_{\ve}\) (corresponding to the homotopy class \(\tilde{\Pi}\))  \st \(\Ind(u_{\ve})\leq k\). Hence, by Theorem \ref{thm interface}, \ref{thm index} and \ref{thm critical set}, there exists
\[V\in\mathbf{C}_{AC}(\tilde{\Pi})\subset \mathbf{C}_{AP}(\Pi)\]
\st \(\Ind(\spt(V))\leq k.\)

\textbf{Acknowledgements.}  I am very grateful to my advisor Prof. Fernando Codá Marques for many helpful discussions and for his support and guidance. The author is partially supported by NSF grant DMS-1811840.

\section{Notations and Preliminaries}
\subsection{Notations}
Here we summarize the notations which will be frequently used later.

\begin{tabular}{ll}
	\([m]\) & the set \(\{1,2,\dots,m\}\)\\
	$\mathcal{H}^s$ & the Hausdorff measure of dimension $s$ \\
	\(A\;\dot\cup\; B\) & the disjoint union of \(A\) and \(B\)\\
	$int(A),\ov{A}$ & the interior of \(A\), the closure of \(A\) (in a topological space)\\
	\(\cm\) & the space of Caccippoli sets in \(M\)\\
	\(\del A\) & the topological boundary of \(A\) (in a topological space) = \(\ov{A}\setminus int(A)\); \(\del\) will also \\ & denote the boundary of a current or the boundary of a cell in a cell-complex.\\
	\(\delst E\) & the reduced boundary of a Caccioppoli set \(E\)\\
	\(\db{S}\) & the current associated to the rectifiable set \(S\)\\
	\(\md{\Si}\) & the varifold associated to the rectifiable set \(\Si\)\\
	\(\norm{V}\) & the Radon measure associated to the varifold \(V\) \\
	\(B^c\) & the complement of \(B\) in \(M\) i.e. \(M\setminus B\)\\
	\(B(p,r)\) & the geodesic ball centered at \(p\) with radius \(r\)\\
	\(A(p,r,R)\) & the annulus centered at \(p\) with radii \(r<R\)\\
	\(d(-,S) \) & distance from a set \(S \subset (M,g)\) \\
	\(\cN_{\rh}(S)\) & the set of points \(x \in (M,g)\) \st \(d(x,S)\leq \rh\) \\
	\(\cT_{\rh}(S)\) & the set of points \(x \in (M,g)\) \st \(d(x,S)=\rh\) \\
	\(H^1(M)\) & the Sobolev space \(\{f \in L^2(M): \text{distributional derivative } \nabla f \in L^2(M) \} \)
\end{tabular}

\subsection{Varifolds}
Here we will briefly discuss the notion of varifold; further details can be found in Simon's book \cite{Sim}. Given a Riemannian \mf $(M^{n+1},g)$, let $G_nM$ denote the Grassmanian bundle of $n$-dimensional hyperplanes over $M$. An $n$\textit{-varifold} in $M$ is a positive Radon measure on $G_nM$. If $V$ is an \(n\)-varifold and $\mathbf{p}:G_nM \ra M$ is the canonical projection map, $\|V\| = \mathbf{p}_{*}V$ is a Radon measure on $M$; $\|V\|(A)=V(\mathbf{p}^{-1}(A))$. The topology on the space of $n$-varifolds is given by the weak* topology i.e. a net $\{V_i\}_{i \in \cI}$ converges to $V$ if and only if
$$\int _{G_nM} f(x, \omega)dV_i(x, \omega) \ra \int _{G_nM} f(x, \omega)dV(x, \omega)$$
for all $f \in C_c(G_nM)$. If \(\Si \subset M\) is \(n\)-rectifiable and \(\tht : \Si \ra [0,\infty)\) is in \(L^1_{\text{loc}}(\Si,\cH^n)\), the \(n\)-varifold \(\mathbf{v}(\Si,\tht)\) is defined by
\[\mathbf{v}(\Si,\tht)(f)=\int_{\Si}f(x,T_x\Si)\tht(x)\; d\cH^n(x)\]
where \(T_x\Si\) is the approximate tangent space of \(\Si\) at $x$ which exists $\cH^n \mres \Si$ a.e. Such a varifold is called a \textit{rectifiable $n$-varifold}. When $\tht$ is the constant function $1$, $\mathbf{v}(\Si,\tht)$ is denoted by $\md{\Si}$; it is called the \textit{varifold associated to $\Si$.}

If $\varphi : M \ra M$ is a \(C^1\) map and $ V$ is an \(n\)-varifold in \(M\), the push-forward varifold $\varphi_{\#}V$ is defined as follows.
$$(\varphi_{\#}V)(f) = \int _{G_n^{+}M}f\left(\varphi(x),D\varphi|_x(\omega)\right)J\varphi(x, \omega)dV(x, \omega)$$
where 
$$J\varphi(x, \omega) = \left(\det\left(\left(D \varphi(x) \big |_{\omega}\right)^t \circ \left(D\varphi(x) \big |_{\omega}\right)\right)\right)^{1/2}$$
is the Jacobian factor and \[G_n^{+}M=\{(x,\om)\in G_nM: J\vp(x,\om)\neq 0\}.\] If \(V=\mathbf{v}(\Si,\tht)\) is a rectifiable \(n\)-varifold, \(\vp_{\#}V=\mathbf{v}(\vp(\Si),\tilde{\tht})\); \(\tilde{\tht}:\vp(\Si)\ra \bbr\) is defined by \[\tilde{\tht}(y)=\sum_{x \in \vp^{-1}(y)\cap \Si}\tht(x).\]
We denote by \(\cV_n(M)\) the closure of the space of rectifiable \(n\)-varifolds in \(M\) \wrt the above varifold weak topology. The \(\bF\) metric on \(\cV_n(M)\) is defined as follows \cite{Pitts}*{page 66}.
\[\bF(V,W)=\sup \{V(f)-W(f): f\in C_c^{0,1}(G_nM), \md{f}\leq 1, \text{Lip}(f)\leq 1\}.\]
For every \(a>0\), the \(\bF\)-metric topology and the varifold weak topology coincide on the set \(\{V\in \cV_n(M):\norm{V}(M)\leq a\}\).

\subsection{Almgren-Pitts min-max theory}
In this subsection, we will recall some of the definitions in the Almgren-Pitts min-max theory; we refer to the papers by Marques and Neves \cite{MN_Willmore,MN_ricci_positive,MN_index,MN_index_2}, Schoen and Simon \cite{SS} and the book by Pitts \cite{Pitts} for more details. To discuss the Almgren-Pitts min-max theory we need to introduce the following spaces of currents. Let \((M^{n+1},g)\) be a closed Riemannian manifold. \(\mathbf{I}_l(M^{n+1};\bbz_2)\) is the space of \(l\)-dimensional mod \(2\) flat chains in \(M\); we only need to consider \(l=n,n+1\). \(\zn\) denotes the space of flat chains \(T \in \mathbf{I}_n(M;\bbz_2)\) \st \(T=\del U\) for some \(U \in \mathbf{I}_{n+1}(M;\bbz_2)\). For \(T \in \cZ_n(M;\bbz_2)\), \(\md{T}\) stands for the varifold associated to \(T\) and \(\norm{T}\) is the radon measure associated to \(\md{T}\). \(\cF\) and \(\bM\) denote the flat and mass norm on \(\mathbf{I}_l(M;\bbz_2)\). When \(l=n+1\), these two norms coincide. The \(\bF\) metric on the space of currents is defined as follows.
\begin{align*}
	&\bF(U_1,U_2)=\cF(U_1,U_2)+\bF(\md{\del U_1},\md{\del U_2}) \text{ if } U_1, U_2 \in \mathbf{I}_{n+1}(M;\bbz_2);\\
	&\bF(T_1,T_2)=\cF(T_1,T_2)+\bF(\md{T_1},\md{ T_2}) \text{ if } T_1, T_2 \in \mathbf{I}_{n}(M;\bbz_2).
\end{align*}
It is proved in \cite{MN_index_2} that the space \(\mathbf{I}_{n+1}(M;\bF;\bbz_2)\) is contractible and the boundary map \(\del:\mathbf{I}_{n+1}(M;\bF;\bbz_2)\ra \cZ_n(M;\bF;\bbz_2)\) is a two sheeted covering map. By the constancy theorem, if \(U_1,U_2 \in \mathbf{I}_{n+1}(M;\bbz_2)\) \st \(\del U_1=\del U_2\), either \(U_1=U_2\) or \(U_1+U_2=\db{M}\).

Let \(X,\; \Pi\) be as in Section 1. The \textit{Almgren-Pitts width} of the homotopy class \(\Pi\) is defined by
\begin{equation}
\bL_{AP}(\Pi)=\inf_{\Ph \in \Pi}\sup_{x\in X}\left\{\bM(\Ph(x))\right\}.\label{2 def AP width}
\end{equation}
A sequence of maps \(\Ph_i:X \ra \cZ_n(M;\bF;\bbz_2)\) in \(\Pi\) is called a \textit{minimizing sequence} if
\begin{equation*}
\limsup_{i \ra \infty}\sup_{x\in X}\left\{\bM(\Ph_i(x))\right\}=\bL_{AP}(\Pi).
\end{equation*}
The \textit{critical set} of a minimizing \seq \(\{\Ph_i\}\), denoted by \(\mathbf{C}\left(\{\Ph_i\}\right)\), is the set of all varifolds \(V \in \cV_n(M)\) \st \(\norm{V}(M)=\bL_{AP}(\Pi)\) and there exist sequences \(\{i_j\}\subset \{i\}\) and \(\{x_j\} \subset X\) \st
\[\lim_{j \ra \infty}\bF\left(V,\md{\Ph_{i_j}(x_{j})}\right)=0.\]
We define \(\mathbf{C}_{AP}(\Pi)\) to be the set of all varifolds \(V\in \cV_n(M)\) \st \(V \in \mathbf{C}\left(\{\Ph_i\}\right)\) for some minimizing sequence \(\{\Ph_i\}\subset \Pi\), \(V\) is a stationary, integral varifold and \(\spt(V)\) is a closed, minimal hypersurface with optimal regularity (i.e. smooth and embedded outside a singular set of Hausdorff dimension \(\leq n-7\)). The theorem of Almgren-Pitts-Schoen-Simon guarantees that \(\mathbf{C}_{AP}(\Pi)\) is non-empty.\label{section 2.3}

\subsection{Allen-Cahn min-max theory}
We now briefly discuss the min-max theory for the Allen-Cahn functional following the papers by Guaraco \cite{Guaraco} and Gaspar-Guaraco \cite{GG1} where further details can be found. Let \(W:\bbr \ra \bbr\) be a smooth, symmetric, double well potential. More precisely, \(W\) has the following properties. \(W\geq 0\); \(W(-t)=W(t)\) for all \(t \in \bbr\); \(W\) has exactly three critical points \(0,\pm 1\); \(W(\pm 1)=0\) and \(W''(\pm 1)>0\) i.e. \(\pm 1\) are non-degenerate minima; \(0\) is a local maximum. The Allen-Cahn energy (with parameter \(\ve>0\)) is given by 
\[E_{\ve}(u)=\int_M\ve\frac{|\nabla u|^2}{2}+\frac{W(u)}{\ve}.\]
As mentioned earlier, 
\[AC_{\ve}(u):=\ve\De u-\ve^{-1} W'(u)=0\]
if and only if \(u\) is a critical point of \(E_{\ve}\). 

Let \(\xt,\; \tilde{\Pi}\) be as in Section 1. The \textit{\(\ve\)-Allen-Cahn width} of the homotopy class \(\tilde{\Pi}\) is defined by
\begin{equation}
\bL_{\ve}(\tilde{\Pi})=\inf_{h \in \tilde{\Pi}} \sup_{x \in \xt} E_{\ve}(h(x)). \label{2 def AC width}
\end{equation}
A sequence of maps \(h_i:\xt \ra H^1(M)\setminus \{0\}\) in \(\tilde{\Pi}\) is called a \textit{minimizing sequence} for \(E_{\ve}\) if 
\[\limsup_{i \ra \infty}\sup_{x \in \xt}E_{\ve}(h_i(x))=\bL_{\ve}(\tilde{\Pi})\]
\(u\) is called a \textit{min-max critical point} of \(E_{\ve}\) (corresponding to the homotopy class \(\tilde{\Pi}\)) if \(u\) is a critical point of \(E_{\ve}\) with \(E_{\ve}(u)=\bL_{\ve}(\tilde{\Pi})\) and
\[\lim_{i \ra \infty} d_{H^1(M)}\left( u, h_i(\xt)\right) = 0 \]
where \(\{h_i\}\) is a minimizing sequence for \(E_{\ve}\) in \(\tilde{\Pi}\).

As \(W\) is an even function, \(E_{\ve}\) is invariant under the \(\bbz_2\) action on \(H^1(M)\) given by \(u \mapsto -u\) i.e. \(E_{\ve}(u)=E_{\ve}(-u)\). Moreover, as proved in \cite{Guaraco}*{Proposition 4.4}, \(E_{\ve}\) satisfies the Palais-Smale condition for bounded sequences. Hence, as explained in \cite{GG1}, if \(\ve>0\) satisfies
\begin{equation}\label{condition}
\bL_{\ve}(\tilde{\Pi})<E_{\ve}(0)=\frac{W(0)}{\ve}\vol(M,g)
\end{equation}
(which holds for \(\ve\) sufficiently small by \eqref{width ineq}), one can apply Corollary 10.5 of \cite{gh} to the \(\bbz_2\)-homotopic family \(\tilde{\Pi}\) to conclude that there exists a min-max critical point \(u_{\ve}\) of \(E_{\ve}\) (corresponding to the homotopy class \(\tilde{\Pi}\)) \st \(\Ind(u_{\ve})\leq k\). (Here \(k\) is the dimension of the parameter space \(\tx.\)) The restriction on \(\ve\) given by \eqref{condition} is due to the fact that the space \(H^1(M)\setminus\{0\}\) is not complete; \eqref{condition} ensures that a minimizing sequence for \(E_{\ve}\) is bounded away from \(0\).

\subsection{Convergence of the phase interfaces}
Let us define \(F:\bbr \ra \bbr\) and the energy constant \(\si\) as follows.
\begin{equation}
F(a)=\int_{0}^{a}\sqrt{W(s)/2}\; ds;\quad \quad \si=\int_{-1}^{1}\sqrt{W(s)/2}\;ds\quad \text{so that}\quad F(\pm 1)=\pm \frac{\si}{2}.\label{2 defn F}
\end{equation}
Let \(u \in C^1(M)\), \(w=F \circ u\). The \(n\)-varifold associated to \(u\) is defined by
\[V[u](A)=\frac{1}{\si}\int_{-\infty}^{\infty}\md{\{w=s\}}(A)\; ds\]
for every Borel set \(A \subset G_nM\). On a closed manifold, if \(AC_{\ve}(u)=0\) and \(u\) is not identically equal to \(\pm 1\), \(\md{u}< 1\) \cite{GG1}*{Lemma 2.2}; in that case, in the definition of \(V[u]\) the integral can be taken over the interval \((-\si/2,\si/2).\)

Building on the work of Hutchinson-Tonegawa \cite{HT}, Tonegawa \cite{t} and Tonegawa-Wickramasekera \cite{TW}, Guaraco \cite{Guaraco} has proved the following theorem.
\begin{thm}[\cite{HT,t,TW,Guaraco}]\label{thm interface}
	Let \(\{u_{i}:M\ra (-1,1)\}_{i=1}^{\infty}\) be a sequence of smooth functions \st
	\begin{itemize}\vspace{-3ex}
		\item[(i)] \(AC_{\ve_i}(u_i)=0\) with \(\ve_i\ra 0\) as \(i \ra \infty\);
		\item[(ii)] There exists \(E_0>0\) and \(I_0\in \mathbb{N}_0\) \st \(E_{\ve_i}(u_i)\leq E_0\) and \(\Ind(u_i)\leq I_0\) for all \(i \in \bbn\).
	\end{itemize}\vspace{-3ex}
Then, there exists a stationary, integral varifold \(V\) \st possibly after passing to a subsequence, \(V[u_i]\ra V\) in the sense of varifolds. Moreover,
\[\norm{V}(M)=\frac{1}{2\si}\lim_{i \ra \infty} E_{\ve_i}(u_i)\]
and \(\spt(V)\) is a closed, minimal hypersurface with optimal regularity.
\end{thm}\vto
The proof of the regularity of the limit-interface depends on the regularity theory of stable, minimal hypersurfaces developed by Wickramasekera \cite{w}.

The upper bound for the Morse index of the limit-interface in the above theorem was proved by Gaspar \cite{Gaspar} and Hiesmayr \cite{H} (when the limit-interface is two sided).

\begin{thm}[\cite{Gaspar,H}]\label{thm index}
	The Morse index of the limit-interface \(\spt(V)\) in the above Theorem \ref{thm interface} is less than or equal to \(I_0\).
\end{thm}\vth
Lastly, we introduce the following definition. \(\mathbf{C}_{AC}(\tilde{\Pi})\) is the set of all stationary, integral \(n\)-varifolds \(V\) \st \(\spt(V)\) is a closed, minimal hypersurface with optimal regularity and \(V\) is the varifold limit of \(V[u_i]\) for some \seq \(\{u_i\}_{i=1}^{\infty}\) \st \(u_i\) is a min-max critical point of \(E_{\ve_i}\) (with \(\ve_i\ra 0\)) corresponding to the homotopy class \(\tilde{\Pi}\). By the discussion of Section 2.4 and Theorem \ref{thm interface}, \(\mathbf{C}_{AC}(\tilde{\Pi})\) is non-empty.\label{section 2.5}

\section{Proof of the width inequality}
In this section we will prove our main Theorem \ref{thm main thm}. Let us fix \(\et >0\). Let \(L=\bL_{AP}(\Pi)\). By the interpolation theorems of Pitts and Marques-Neves \cite{MN_ricci_positive,MN_index_2} there exists  $\Ph:X \ra \mathcal{Z}_n(M^{n+1}; \bM;\bbz_2)$ such that
\begin{equation}
	\sup_{x \in X}\{\bM(\Ph(x))\}<L+\et. \label{eq mass of Phi is bounded}
\end{equation}
We choose \(\tph:\tx \ra \cm\) which is a lift of \(\Ph\) i.e. for all \(x \in \xt\),
\[\db{\del^{*}\tph(x)}=\del \db{\tph(x)}=\Ph(\pi(x)).\]
\(\tph\) is \(\bbz_2\)-equivariant i.e. if \(T:\tx \ra \tx\) is the deck transformation, \(\db{\tph(x)}+\db{\tph(T(x))}=\db{M}\) for all \(x \in \tx\).
\subsection{Approximation of a Caccippoli set by open sets with smooth boundary}
In this subsection, following the book by Giusti \cite{giu} and the paper by Miranda-Pallara-Paronetto-Preunkert \cite{MPPP}, we briefly discuss the fact that a Caccioppoli set can be approximated by open sets with smooth boundary. We begin with the following theorem.
\begin{thm}[\cite{giu}*{Theorem 1.17}, \cite{MPPP}*{Proposition 1.4}]\label{thm 3.2}
	Let \(E\in \cm\). There exists a sequence of smooth functions \(\{f_j:M \ra \bbr\}_{j=1}^{\infty}\) \st \(0\leq f_j \leq 1\) for all \(j\) and
	\[\lim_{j \ra \infty}\int_M|f_j -\ch_E|\;d\cH^{n+1} =0 \quad \text{ and } \quad \int_M|D\ch_E|=\lim_{j\ra \infty}\int_M|Df_j|.\]
\end{thm}\vspace{-4ex}
Following \cite{giu}*{Proof of Theorem 1.24}, for \(t\in(0,1)\), let us define \(E_{j,t}=\{f_j>t\}\). Then,
\begin{equation}
|f_j-\ch_E|> t \text{ on }  E_{j,t}\setminus E \quad \text{  and  } \quad |f_j-\ch_E|\geq 1-t \text{ on } E \setminus E_{j,t}  \label{eq 3.2.1}
\end{equation} 
which implies
\begin{equation}
\int_M|f_j-\ch_{E}|\;d\cH^{n+1}\geq \min \{t, 1-t\}\int_M |\ch_{E_{j,t}}-\ch_E|\;d\cH^{n+1}. \label{eq 3.2.2}
\end{equation}
Hence, for all \(t\in (0,1)\),
\begin{align}
&\lim_{j\ra \infty} \int_M |\ch_{E_{j,t}}-\ch_E|\;d\cH^{n+1} =0 \label{eq: convergence in L1 }\\
&\Longrightarrow \int_M|D\ch_E| \leq \liminf_{j \ra \infty}\int_M |D\ch_{E_{j,t}}|.\label{eq 3.2.3}
\end{align}
\tf using Theorem \ref{thm 3.2}, the co-area formula for the BV function and \eqref{eq 3.2.3} we obtain the following inequalities.
\[\int_M |D\ch_E|=\lim_{j\ra \infty}\int_M|Df_j|\geq \int_{0}^{1}\left(\liminf_{j \ra \infty}\int_M |D\ch_{E_{j,t}}|\right)dt\geq \int_M|D\ch_E|.\]
This implies
\begin{equation}
\liminf_{j \ra \infty}\int_M |D\ch_{E_{j,t}}|= \int_M|D\ch_E| \text{ for a.e. } t\in (0,1). \label{eq: convergence of liminf}
\end{equation}
We choose \(t_0\in(0,1)\) \st \(t_0\) is a regular value of \(f_j\) for all \(j\) and \eqref{eq: convergence of liminf} holds for \(t=t_0\). Further, possibly after passing to a subsequence, we can assume that 
\begin{equation}
\lim_{j \ra \infty}\int_M |D\ch_{E_{j,t_0}}|= \int_M|D\ch_E|. \label{eq: convergence of lim}
\end{equation}
Let us define \(E_j=\ov{E}_{j,t_0}\). Since
\[E_{j,t_0} \subset \ov{E}_{j,t_0} \subset E_{j,t_0} \cup \{f_j=t_0\}, \]
we have \(\cH^{n+1}\left(\ov{E}_{j,t_0} \setminus E_{j,t_0} \right).\) From \eqref{eq: convergence in L1 } and \eqref{eq: convergence of lim} we conclude that
\begin{equation}
\ch_{E_j} \ra \ch_E \text{ in } L^1(M) \text{  and  } \lim_{j \ra \infty}\int_M |D\ch_{E_{j}}|= \int_M|D\ch_E|. \label{eq: convergence in F}
\end{equation}
By \cite{Pitts}*{2.1(18)(f), page-63}, \eqref{eq: convergence in F} implies that \([\![\del^{*} E_j]\!]\) converges to \([\![\del^{*}E]\!]\) in \(\bF\).

Let us fix \(p \in M\) and \(R>0\). Using \eqref{eq: convergence in F},
\[\lim_{j\ra \infty}\int_{0}^{R}\bigg(\int_{\del B(p,t)}|\ch_{E_j}-\ch_E| \; d\cH^n\bigg)dt = \lim_{j \ra \infty}\int_{B(p,R)}|\ch_{E_j}-\ch_E| \; d\cH^{n+1}=0.\]
Hence, there exists a subsequence \(\{\ch_{E_{j_s}}\} \subset \{\ch_{E_j}\}\) \st 
\[\lim_{s \ra \infty}\int_{\del B(p,t)}|\ch_{E_{j_s}}-\ch_E| \; d\cH^n =0 \text{ for a.e. } t \in (0,R).\]

Next, we define \(F_j=\ov{\{f_j< t_0\}}\) and \(F=M\setminus E\). Then, \(E_j \cap F_j\subset \{f_j=t_0\} \) and \(E_j \cup F_j=M\). Therefore, \([\![E_j]\!]+[\![F_j]\!]=[\![M]\!]\), \(\del^{*}E_j=\del^{*}F_j\) and \(\ch_{F_j} \ra \ch_F\) in \(L^1(M)\). We note that \(\del E_j,\; \del F_j \subset \{f_j=t_0\}\). As the reduced boundary is a subset of the topological boundary, we also have \(\del^{*}E_j=\del^{*}F_j\subset \{f_j=t_0\}.\) Since \(\{f_j=t_0\}\) is a smooth, closed \hy in \(M\), for each \(a \in \{f_j=t_0\}\) there exist \(\rh>0\) and co-ordinates \(\{x_1,x_2,\dots , x_{n+1}\}\) on \(B(a,\rh)\) \st \(x_i(a)=0\) for all \(i\) and
\[\left(B(a,\rh)\cap \{f_j=t_0\}\right)=\{x \in B(a,\rh):x_{n+1}=0\}.\]
Let \(G_1=\{x \in B(a,\rh):x_{n+1}<0\}\) and \(G_2=\{x \in B(a,\rh):x_{n+1}>0\}\) so that
\[\left(B(a,\rh)\setminus \{f_j=t_0\}\right)=G_1\; \dot\cup \; G_2 .\]
\(G_1\) and \(G_2\) are connected open sets. We have the following three mutually exclusive cases.
\begin{itemize}\vspace{-3ex}
	\item[(1)] \(f_j >t_0\) on both \(G_1\) and \(G_2\); this implies \(a \in int(E_j)\setminus F_j\);
	\item[(2)] \(f_j <t_0\) on both \(G_1\) and \(G_2\); this implies \(a \in int(F_j) \setminus E_j\);
	\item[(3)] \(f_j >t_0\) on one of \(G_1\) and \(G_2\), and \(f_j <t_0\) on the other; this implies \(a \in E_j \cap F_j\).
\end{itemize}\vspace{-2ex}
From the above three cases we can conclude that \(\del^{*}E_j=\del^{*}F_j= E_j \cap F_j= \del E_j=\del F_j\) which is a smooth, closed, embedded \hy in \(M\). Indeed, the set of points \(a \in \{f_j=t_0\}\) for which the above item (3) holds is both open and closed in \(\{f_j=t_0\}\); hence it is the union of certain connected components of \(\{f_j=t_0\}.\)

From the above discussion we arrive at the following proposition.
\begin{pro}\label{pro: modified smooth approximation}
	Let \(E \in \cm\) and \(F=M \setminus E\). Then, for each \(j \in \bbn\) there exist closed sets \(E_j,\; F_j \in \cm\) \st the followings hold.
	\begin{itemize}\vspace{-3ex}
		\item[(i)]\([\![E_j]\!]+[\![F_j]\!]=[\![M]\!]\) and \(M=E_j \cup F_j\).
		\item[(ii)] \(\ch_{E_{j}} \ra \ch_{E}\) and \(\ch_{F_{j}} \ra \ch_{F}\) in \(L^1(M)\).
		\item[(iii)] \(\del^{*}E_j=\del^{*}F_j= E_j \cap F_j= \del E_j=\del F_j\) is a smooth, closed, embedded \hy in \(M\).
		\item[(iv)] \([\![\del^{*} E_j]\!]=[\![\del^{*} F_j]\!]\) converges to \([\![\del^{*}E]\!]=[\![\del^{*}F]\!]\) in \(\bF\).
		\item[(v)] For every \(p \in M\) and \(R>0\) there exist subsequences \(\{\ch_{E_{j_s}}\} \subset \{\ch_{E_j}\}\) and \(\{\ch_{F_{j_s}}\} \subset \{\ch_{F_j}\}\) \st 
		\[\lim_{s \ra \infty}\int_{\del B(p,t)}|\ch_{E_{j_s}}-\ch_E| \; d\cH^n =0 \; \text{ and }\; \lim_{s \ra \infty}\int_{\del B(p,t)}|\ch_{F_{j_s}}-\ch_F| \; d\cH^n =0\]
		for a.e.  \(t \in (0,R).\)
	\end{itemize}\vspace{-2ex}
\end{pro}\vspace{-2ex}

\subsection{Preliminary constructions}
Let \(D\) be a countable, dense subset of \(M\) and \(\text{inj}(M)\) be the injectivity radius of \((M,g)\). We consider
\[\mathscr{B}=\{B(p,t): p\in D, t \in (0, \text{inj}(M))\cap \mathbb{Q}\}\]
which is also a countable set. Let us assume that \(M\) is isometrically embedded in some Euclidean space \(\bbr^m\). We have the following theorem which is a consequence of Sard's theorem. 
\begin{thm}(\cite{Nic}*{Corollary 1.25})
	Let \(\Si\) be a closed submanifold of \(\bbr^m\). For \(v \in \bbr^m\) we define \(f_v: \bbr^m \ra \bbr\) by \(f_v(x)=\langle x, v\rangle\). Then, there exists a generic set \(V \subset \bbr^m\) (depending on \(\Si\)) \st for all \(v \in V\), \(f_v\vert_{\Si}\) is a Morse function on \(\Si\). 
\end{thm}\vspace{-2ex}
By the above theorem, there exists \(\om \in \bbr^m\) \st \(f_{\om}\vert_{M}\) and \(f_{\om}\vert_{\del B}\) are Morse functions for all \(B \in \mathscr{B}\). By composing scaling and translation with \(f_{\om}\), we can assume that \(f_{\om}(M)=[1/3,2/3]\). From now on, whenever we will consider \(f_{\om}\), it will be assumed that \(f_{\om}: M \ra [1/3,2/3]\). Let us choose \(r_0\in (0, \text{inj}(M))\) \st
\begin{itemize}\vspace{-3ex}
	\item \(\cH^n\left(f_{\om}^{-1}(t)\cap \overline{B}(p,r_0)\right)<\et/2\) for all \(t \in [1/3,2/3],\; p\in M\);
	\item \(\bM \left(\Ph(x)\mres \overline{B}(p,r_0)\right)<\et\) for all \(x \in X,\; p \in M\);
\end{itemize}\vspace{-2ex}
where \(\Ph\) is as chosen at the beginning of Section 3. Such a choice of \(r_0\) is possible because of the `no concentration of mass property'(\cite{MN_ricci_positive}). We choose \(p_i \in D\) \st
\begin{equation}
	M=\bigcup_{i=1}^{I}B\left(p_i,r_0/4\right) ; \qquad \bbb_i^0 :=B(p_i,r_0).  \label{def: bbb0}
\end{equation}
Hence, in particular, we have
\begin{align}
&\cH^n\left(f_{\om}^{-1}(t)\cap \overline{\bbb_i^0}\right)<\et/2 \;\;\forall t \in [1/3,2/3], i \in [I];\nonumber\\
&\bM \left(\Ph(x)\mres \overline{\bbb_i^0}\right)<\et\;\; \forall x \in X, i \in [I]. \label{eq mass of Phi in Biz} 
\end{align}
\begin{lem}
	There exists \(r_1 \in (r_0/2, 3r_0/4)\cap \mathbb{Q}\), \( \delta \in (0, r_0/8)\) \st 
	\[\bM\left( \Ph(x)\mres \ov{A}(p_i, r_1-2\delta, r_1+\delta)\right)<\frac{\et}{I}\]
	for all \(x \in X\) and \(i \in [I].\)\label{lem: mass of phi in annulus}
\end{lem}\vspace{-4ex}
\begin{proof}
	By the compactness of \(X\), \te \(\{x_j\}_{j=1}^J \subset X\) \st 
	\begin{equation}
	\forall x \in X, \; \exists j \in [J] \text{ such that } \bM \left( \Ph (x) - \Ph(x_j)\right) < \frac{\et}{2I}.\label{eqn: xj's}
	\end{equation} 
	For \(i \in [I], j\in [J]\) we define
	\[\cR_{ij}=\left\{ r' \in \left(r_0/2,3r_0/4\right):\bM\left(\Ph(x_j)\mres \del B(p_i, r')\right)<\frac{\et}{2I}\right\}.\]
	We note that 
	\[\left(r_0/2,3r_0/4\right)\setminus \cR_{ij} \text{ is finite } \forall\; i,j.\]
	Hence, we can choose 
	\begin{equation}
	r_1 \in \Big(\bigcap_{i\in [I], j\in [J]}\cR_{ij}\Big)\cap\left(r_0/2,3r_0/4\right)\cap \mathbb{Q}. \label{eqn: choosing r1}
	\end{equation} 
	\tf we have
	\[\bM\left(\Ph(x_j)\mres \del B(p_i,r_1)\right)< \frac{\et}{2I} \quad \forall \; i,j.\]
	Hence, we can choose \(\delta \in (0,r_0/8)\) \st
	\[\bM\left(\Ph(x_j)\mres \ov{A}(p_i,r_1-2\delta,r_1+\delta)\right)<\frac{\et}{2I} \quad \forall \; i,j;\]
	which implies (by \eqref{eqn: xj's})
	\[\bM\left( \Ph(x)\mres \ov{A}(p_i, r_1-2\delta, r_1+\delta)\right)<\frac{\et}{I} \quad \forall \; x \in X \; i \in [I].\]
\end{proof}
\begin{lem}
	There exists \(\delta\in (0,r_0/8)\) such that
	\[\cH^n\left(f_{\om}^{-1}(t)\cap \ov{A}(p_i, r_1-2\de, r_1+\delta)\right)<\frac{\et}{2I}\]
	for all \(t \in [1/3,2/3]\), \(i \in [I]\). \label{lem: mass of initial morse fn level set}
\end{lem}\vspace{-4ex}
\begin{proof}
	We assume by contradiction that there exist \(\{t_j\}_{j=1}^{\infty}\subset [1/3,2/3]\) and \(\{d_j\}_{j=1}^{\infty}\subset \bbr^{+}\) \st \(d_j \ra 0\) and for some \(i \in [I]\)
	\[\cH^n\left(f_{\om}^{-1}(t_j)\cap \ov{A}(p_i, r_1-2d_j, r_1+d_j)\right)\geq\frac{\et}{2I}\]
	holds for all \(j \in \bbn\). Without loss of generality we can assume that \(t_j \ra t\). Denoting
	\[\mu_j=\cH^n\mres f_{\om}^{-1}(t_j) \qquad \text{ and } \qquad \mu=\cH^n\mres f_{\om}^{-1}(t),\]
	we have that \(\mu_j\) weakly converges to \(\mu\) (in the sense of radon measure). Let us fix \(l \in \bbn\). There exists \(j_0 \in \bbn\) \st \(d_j \leq l^{-1}\) for all \(j \geq j_0\). Hence, 
	\[\mu_j(\ov{A}(p_i, r_1-2l^{-1}, r_1+l^{-1}))\geq \frac{\et}{2I}\]
	for all \(j \geq j_0\). Therefore,
	\[\mu(\ov{A}(p_i, r_1-2l^{-1}, r_1+l^{-1}))\geq \limsup_{j \ra \infty} \mu_j(\ov{A}(p_i,r_1-2l^{-1}, r_1+l^{-1} ))\geq \frac{\et}{2I}.\]
	This holds for all \(l \in \bbn\). Denoting \(A_l = \ov{A}(p_i, r_1-2l^{-1}, r_1+l^{-1})\) we have
	\[A_{l+1}\subset A_l \quad \text{ and } \quad \bigcap_{l=1}^{\infty}A_l= \del B(p_i, r_1).\]
	This implies
	\[\cH^n\left(f_{\om}^{-1}(t)\cap \del B(p_i, r_1)\right)=\mu(\del B(p_i, r_1))=\lim_{l \ra \infty} \mu(A_l) \geq \frac{\et}{2I}.\]
	However, this is not possible. Indeed, \(r_1 \in \mathbb{Q}\) (Lemma \ref{lem: mass of phi in annulus}) and by the choice of \(\om\), \(f_{\om}\vert_{\del B(p_i, r_1)}\) is a Morse function on $\del B(p_i, r_1)$; hence \( \cH^n\left(f_{\om}^{-1}(t)\cap \del B(p_i, r_1)\right)\) must be \(0\).
\end{proof}
We can assume that the two \(\delta\)'s appearing in Lemma \ref{lem: mass of phi in annulus} and Lemma \ref{lem: mass of initial morse fn level set} are the same. Next, we modify \(f_{\om}\) near the points where it achieves local maxima or local minima (which are not global maxima or minima) to get another Morse function \(f:M \ra [1/3, 2/3]\) \st \(f\) has no non-global local maxima or local minima and for all \(t \in [1/3,2/3]\), \(i \in [I]\)
\begin{equation}
\label{eq: mass of modified morse fn level set} 
\cH^n\left(f^{-1}(t)\cap \overline{\bbb^0_i}\right)<\et \quad \text{ and } \quad \cH^n\left(f^{-1}(t)\cap \ov{A}(p_i, r_1-2\de, r_1+\delta)\right)<\frac{\et}{I}.
\end{equation}
Let us introduce the following notation which will be used later.
\begin{align}
&\bbb_i^1 = B(p_i,r_1); \qquad \cA_1=\bigcup_{i=1}^IA(p_i,r_1-2\de, r_1+\de); \nonumber \\
&\cA_2=\bigcup_{i=1}^IA\left(p_i,r_1-3\de/2, r_1+\de/2\right)  ; \qquad \cA=\bigcup_{i=1}^IA(p_i,r_1-\de, r_1). \label{eq: notation for ball and annulus}
\end{align}
\subsection{Cell complex structure on the parameter space}
For \(l \in \bbn\), \(\mathscr{I}[l]\) is the cell complex on \(\mathscr{I}=[0,1]\) whose \(0\)-cells are
\[[0],\; [l^{-1}], \dots , [1-{l^{-1}}],\; [1]\]
and \(1\)-cells are 
\[[0,l^{-1}],\;[l^{-1},2l^{-1}],\dots, [1-l^{-1},1].\]
\(\mathscr{I}^m[l]\) denotes the cell complex on \(\mathscr{I}^m=[0,1]^m\) whose cells are \(\al_1 \otimes \al_2 \otimes \dots \otimes \al_m\) where each \(\al_j \in \mathscr{I}[l]\). We note that \(\mathscr{I}^m[1]\) is the standard cell complex on \(\mathscr{I}^m\). By abuse of notation, we will identify a cell \(\al_1 \otimes \al_2 \otimes \dots \otimes \al_m\) with its support \(\al_1 \times \al_2 \times \dots \times \al_m \subset \mathscr{I}^m\). Similarly, if \(Y\) is a subcomplex of \(\mathscr{I}^m[1]\), \(Y[l]\) is the union of all the cells in \(\mathscr{I}^m[l]\) whose support is contained in \(Y\). If \(\cY\) is a cell complex, \(\cY_p\) will denote the set of all \(p\)-cells in \(\cY\); if \(\al,\be \in \cY\) \st \(\be\) is a face of \(\al\) (in the definition of the face, we do not insist that \(\dim(\be)<\dim(\al)\) or \(\dim(\be)=\dim(\al)-1\)), we use the notation \(\be \prec \al\).
\newcommand{\il}{\mathscr{I}[l]} \newcommand{\iml}{\mathscr{I}^m[l]}\newcommand{\sci}{\mathscr{I}}

If \(\la=[il^{-1},(i+1)l^{-1}] \in \il_1\), there exists a canonical map
\begin{equation}
\De_{\la}:\mathscr{I} \ra \mathscr{I};\qquad \De_{\la}(t)=(i+t)l^{-1}\label{eq defining delta lambda}
\end{equation}
such that \(\De_{\la}:\mathscr{I} \ra \la\) is a homeomorphism. Similarly, if \(\al = \al_1 \otimes \al_2 \otimes \dots \otimes \al_m \in \iml_p\), there exists a canonical map \(\De_{\al}:\mathscr{I}^p \ra \sci^m\) defined as follows. There exist precisely \(p\) indices
\[j_1<j_2<\dots <j_p \;\text{  such that  } \dim(\al_{j_s})=1 \; \forall  s=1,2,\dots,p.\]
We define 
\begin{equation}
\left(\De_{\al}(t_1,t_2, \dots, t_p)\right)_j=
\begin{cases}
\De_{\al_{j_s}}(t_s) & \text{if } j =j_s,\\
\al_j & \text{if } j \notin \{j_1,j_2,\dots,j_p\} \Leftrightarrow \dim(\al_j)=0.
\end{cases}
\end{equation}
\(\De_{\al}:\mathscr{I}^p \ra \al\) is a homeomorphism. Let \(D_{\al}:\al \ra \sci^p\) be the inverse of \(\De_{\al}\).

If \(\be \prec \al\) and \(D_{\al}(\be)=\varrho \prec \sci^p\), then the following compatibility relation holds.
\begin{equation}
	\De_{\al}\circ \De_{\varrho}=\De_{\be}. \label{eq compatibility}
\end{equation}

We recall from Section 1 that \(X\) is a subcomplex of \(\sci^N[1]\) for some \(N\) and \( \pi: \tx \ra X\) is a double cover. For \(l\in\bbn\), \(X[l]\) is the cell complex on \(X\) as defined above. \(\tx[l]\) denotes the cell complex on \(\tx\) whose cells are pre-images of the cells of \(X[l]\) via the map \(\pi\). We choose \( K \in \bbn\) \st the followings hold. (Here \(\Ph\) and \(\tilde{\Ph}\) are as chosen at the beginning of Section 3.)
\begin{align}
&\text{If } x_1,x_2 \in X[K]_0 \text{ belong to a common cell in } X[K],\; \bM\left(\Ph(x_1)-\Ph(x_2)\right)<\et. \label{cond: estimating mass norm}\\
&\text{If } y_1,y_2 \in \tx[K]_0 \text{ belong to a common cell in } \tx[K],\; \cH^{n+1}(\tilde{\Ph}(y_1)\De\tilde{\Ph}(y_2))<\frac{\et\de}{I}; \label{cond: estimating flat norm}
\end{align}
where \(\de\) is as in Section 3.2 (Lemma \ref{lem: mass of phi in annulus}, \eqref{eq: mass of modified morse fn level set}). 

Let \(\{c_q:q \in [Q]\}\) be all the cells of \(X[K]\) indexed in such a way that \(\dim(c_{q_1})\leq \dim(c_{q_2})\) if \(q_1 \leq q_2\). Let \(\{e_q,f_q:q \in [Q]\}\) be the cells of \(\xtk\) so that \(\pi^{-1}(c_q)=e_q \;\dot\cup\; f_q\). Since \(c_q\) is contractible, \(\pi|_{e_q}:e_q \ra c_q\) and \(\pi|_{f_q}:f_q \ra c_q\) are homeomorphisms. Let us introduce the following notation (\(d=\dim(c_q)\)).
\begin{align}
 \De_{c_q}=\De_{q}:\sci^{d} \ra c_q \quad&;\quad D_{c_q}=D_{q}:c_q \ra \sci^{d}; \nonumber\\
\left(\pi|_{e_q}\right)^{-1}\circ \De_q=\De^1_q:\sci^{d} \ra e_q \quad &; \quad D_q \circ \left(\pi|_{e_q}\right)= D^1_q:e_q \ra \sci^{d}; \label{eq canonical maps}\\
\left(\pi|_{f_q}\right)^{-1}\circ \De_q=\De^2_q:\sci^{d} \ra f_q \quad &; \quad D_q \circ \left(\pi|_{f_q}\right)=D^2_q:f_q \ra \sci^{d}. \nonumber
\end{align}

\subsection{Construction of an almost smooth sweepout}
The next proposition follows from Proposition \ref{pro: modified smooth approximation}.
\begin{pro}
	\label{pro: seq of good approximations}
	There exists a \seq \(\{\tph_j : \tx[K]_0 \ra \cm\}_{j=1}^{\infty}\) \st the followings hold for all \(j \in \bbn\) and \(x \in \tx[K]_0\).
\begin{itemize}\vspace{-3ex}
	\item[(i)] \(\pht_j(x)\) is a closed subset of \(M\).
	\item[(ii)] \(\db{\pht_j(x)}+\db{\pht_j(T(x))}=\db{M}\) and \(M= \pht_j(x) \cup \pht_j(T(x))\).
	\item[(iii)] As \(j \ra \infty\), \(\ch_{\pht_j(x)} \ra \ch_{\pht(x)}\) in \(L^1(M)\).
	\item[(iv)] \(\del\pht_j(x)=\del\pht_j(T(x))=\pht_j(x)\cap \pht_j(T(x))=\delst\pht_j(x)=\delst\pht_j(T(x))\) is a smooth, closed hypersurface in \(M\).
	\item[(v)] Defining \(\Ph_j(x)= \delst\pht_j(x)\), as \(j \ra \infty\), \(\db{\Ph_j(x)} \ra \Ph(\pi(x))\) in \(\bF\).
	\item[(vi)] For all \(i \in [I]\), 
	\[\lim_{j \ra \infty}\int_{\del B(p_i,t)}|\ch_{\pht_j(x)}-\ch_{\pht(x)}|\;d\cH^n=0\]
	for a.e. \(t \in (0,r_1)\).
\end{itemize}\vspace{-2ex}
\end{pro}
 We will approximate \(\tph\) by a discrete and `almost smooth' sweepout \(\pst:\xt[KI]_0 \ra \cm\) which will be constructed using the \(\pht_j\)'s. The construction is motivated by the interpolation theorems of Almgren \cite{alm_article}, Pitts \cite{Pitts}*{4.5}, Marques-Neves \cite{MN_Willmore}*{Theorem 14.1} and Chambers-Liokumovich \cite{CL}*{Lemma 6.2}. The construction is divided into three parts.

\textbf{Part 1.} We recall that \(\{e_q,f_q:q \in [Q]\}\) are the cells of \(\xtk\). For \(q\in [Q]\) we define the following collection of balls
\begin{equation}
\label{eq: defining the colloection of balls}
\mathscr{B}(q)=\{B(p_i,r_i(q)):i \in [I]\}
\end{equation}
where \(r_i(q)\in (r_1-\de,r_1)\) (\(\de\) is as in Section 3.2) and \(r_i(q)\) is chosen inductively so that the following conditions are satisfied.
\begin{itemize}\vspace{-3ex}
	\item[(i)] \(\norm{\Ph(\pi(x))}(\del B(p_i,r_i(q)))=0 \;\;\forall\; x \in (e_q)_0.\)
	\item[(ii)] \(\del B(p_i,r_i(q))\) is transverse to \(\Ph_j(x)\) for all \(x \in (e_q)_0\) and \(j \in \bbn.\)
	\item[(iii)] \(\del B(p_i,r_i(q))\) is transverse to \(\del B(p_s,r_s(q))\) for all \(s<i\).
	\item[(iv)] \(\del B(p_i,r_i(q))\) is transverse to  \(\del B(p_j,r_j(q'))\) for all \(q' < q\) and \(j \in [I]\).
	\item[(v)] \(\ch_{\pht_j(x)}\) converges to \(\ch_{\pht(x)}\) in \(L^1(M, \cH^n\mres \del B(p_i,r_i(q)))\) for all \(x \in (e_q)_0\cup (f_q)_0\)\footnote{For this item we need to use item (vi) of Proposition \ref{pro: seq of good approximations}.}.
	\item[(vi)] If \(m=\dim(e_q)=\dim(f_q)\),
	\[\int_{\del B(p_i, r_i(q))}|\ch_{\pht(x)}-\ch_{\pht(x')}| \; d\cH^{n}< \frac{2^{2m}\et}{I}\]
	for all \(x,x' \in (e_q)_0\) and for all \(x,x' \in (f_q)_0\)\footnote{For this item we need to use \eqref{cond: estimating flat norm}.}.
\end{itemize}\vspace{-2ex}
Next we choose
\begin{equation*}
r_1>r>\max\{r_i(q):i \in [I], q \in [Q]\}
\end{equation*}
\st the followings hold.
\begin{align}
&\norm{\Ph(\pi(x))}(\del B(p_i,r))=0\; \forall \; x \in \xtk_0,\; i \in [I];\label{eq defining r}\\
&\norm{\Ph_j(x)}(\del B(p_i,r))=0\; \forall \; j \in \bbn,\; x \in \xtk_0,\; i \in [I].\label{eq defining r bis}
\end{align}
We introduce the notation
\begin{equation}
B_i(q)=B(p_i,r_i(q)) \quad;\quad \bbb_i=B(p_i,r). \label{eq defining biq and bbbi}
\end{equation}
We note that by \eqref{def: bbb0}, for all \(q \in [Q]\),
\[\bigcup_{i \in [I]}B_i(q)=M = \bigcup_{i \in [I]}\bbb_i.\]
\textbf{Part 2.} Let us introduce the following definitions.
\begin{equation*}
\mathscr{R}_1=\{\bbb_i:i \in [I]\}\cup \{M \setminus\ov{\bbb}_i: i \in [I]\},
\quad
\mathscr{R}_2=\{U_1\cap U_2 \cap \dots\cap U_s:s\in \bbn \text{ and each } U_j \in \mathscr{R}_1\},
\end{equation*}
\begin{equation*}
\mathscr{R}=\{V_1\cup V_2\cup \dots \cup V_t: t \in \bbn \text{ and each }V_j \in \mathscr{R}_2\}
\end{equation*}
\(\mathscr{R}_1, \mathscr{R}_2\) and \(\mathscr{R}\) are finite sets. In a topological space for any two sets \(A\) and \(B\), 
\begin{equation}
\del (A \cap B), \;\del(A \cup B) \subset \del A \cup \del B.\label{eq boundary}
\end{equation}
Hence, for any \(R \in \mathscr{R}\),
\[\del R \subset \bigcup_{i \in [I]} \del \bbb_i \Longrightarrow \norm{\Ph(\pi(x))}(\del R)=0\; \forall \; x \in \xtk_0.\]
Moreover, \(R\) is an open subset of \(M\). \tf by Proposition \ref{pro: seq of good approximations} item (v), as \(j \ra \infty\)
\begin{equation}
	\norm{\Ph_j(x)}(R) \lra \norm{\Ph(\pi(x))}(R)\; \forall \; x \in \xtk_0. \label{eq phi-j-R converges to phi-R }
\end{equation}
We also note that \(M \in \mathscr{R}\) as \(\cup_{i \in [I]}\bbb_i=M.\)
\begin{pro}\label{pro existence of alpha}
	There exists \(\ga\in \bbn\) \st the followings hold.
\begin{itemize}\vspace{-3ex}
	\item[(i)] $\big|\norm{\Ph_{\ga}(x)}(R)-\norm{\Ph(\pi(x))}(R)\big|<\et$ for all \(x \in \xtk_0\) and \(R \in \mathscr{R}\).
	\item[(ii)] \(\cH^n(\Ph_{\ga}(x)) < L+2\et \; \forall  x \in \xtk_0.\)
	\item[(iii)] \(\cH^n(\Ph_{\ga}\cap \ov{\bbb_i^0})<\et \; \forall x \in \xtk_0\).
	\item[(iv)] \(\cH^n(\Ph_{\ga}\cap \ov{\cA_1})<\et \; \forall x \in \xtk_0\).
	\item[(v)] For \(q \in [Q]\) and \(i \in [I]\), if \(\dim(e_q)=\dim(f_q)=m\),
	\begin{equation}
		\int_{\del B_i(q)}|\ch_{\pht_{\ga}(x)}-\ch_{\pht_{\ga}(x')}| \; d\cH^{n}< \frac{2^{2m}\et}{I} \label{eq boundary sphere has small mass}
	\end{equation}
	if \(x,x' \in (e_q)_0\) or if \(x,x' \in (f_q)_0\).
\end{itemize}\vspace{-2ex} 
\end{pro}
\begin{proof}
    (i) follows from \eqref{eq phi-j-R converges to phi-R }. Since, \(M \in \mathscr{R}\), (ii) follows from \eqref{eq mass of Phi is bounded} and (i). To obtain (iii) we note that Proposition \ref{pro: seq of good approximations} item (v) and \eqref{eq mass of Phi in Biz} imply for each \(x \in \xtk_0\),
	\[\limsup_{j \ra \infty}\norm{\Ph_j(x)}(\ov{\bbb_i^0})\leq \norm{\Ph(\pi(x))}(\ov{\bbb_i^0})< \et. \]
	Similarly, (iv) follows from Proposition \ref{pro: seq of good approximations} item (v) and Lemma \ref{lem: mass of phi in annulus}. Finally, item (v) follows from items (v) and (vi) of the definition of \(r_i(q)\).
\end{proof}
Next we define
\[\mathscr{S}=\left\{\Ph_{\ga}(x) \cap \del B_i(q): x \in \xtk_0, i \in [I], q \in [Q]\right\} \bigcup \Big\{\del B_i(q) \cap \del B_j(q'): i,j \in [I]; q,q' \in [Q]\Big\}\]
where \(\ga\) is as in Proposition \ref{pro existence of alpha}. Let \(\cS\) be the closed subset of \(M\) which is the union of all the elements of \(\mathscr{S}\). By the transversality assumptions in the definition of \(r_i(q)\), each non-empty element of \(\mathscr{S}\) is a smooth, closed, co-dimension \(2\) submanifold of \(M\). Hence, for all \(\Si \in \mathscr{S}\)
\[\cH^n(\cT_{\rh}(\Si))=O(\rh)\]
where the constants in \(O(\rh)\) depend on \(\Si\). \tf there exist constants \(C,\; \rh_0\) depending on the submanifolds contained in the set \(\mathscr{S}\) \st
\begin{equation}
	\cH^n(\cT_{\rh}(\cS)) \leq C\rh \quad \forall \rh \leq \rh_0. \label{eq area of equidistant hypersurface}
\end{equation}

\textbf{Part 3.} \newcommand{\xit}{\tilde{\Xi}}\newcommand{\xtki}{\xt[KI]_0}\newcommand{\eqiz}{e_q[I]_0}\newcommand{\fqiz}{f_q[I]_0}\newcommand{\eqz}{(e_q)_0}\newcommand{\fqz}{(f_q)_0}
Let \(\cK(M)\) be the set of all closed subsets of \(M\). We are now going to prove that there exists a discrete, `almost smooth' sweepout \(\pst :\xt[KI]_0 \ra \cK(M)\) which approximates \(\pht: \xt \ra \cm.\) With some additional work it is possible to ensure that for all \(v \in \xtki\), \(\pst(v) \in \cm\); however we do not need this fact to prove Theorem \ref{thm main thm}. Let us introduce the following notation. Let \(\mathbb{K},\; \mathbb{K}'\) be closed subsets of \(M\). \(B\) is a normal geodesic ball and \(A=M\setminus \overline{B}.\) We define
\begin{equation}
\Om(\mathbb{K},\mathbb{K}';B) = (\mathbb{K} \cap \ov{A} ) \cup (\mathbb{K}' \cap \ov{B}) \label{eq defining Omega operation}
\end{equation}
which is also a closed subset of \(M\). This definition is motivated by the construction of ``nested sweepouts" by Chambers and Liokumovich \cite{CL}*{Section 6}. In the following proposition and in its proof, we will use various notations which were introduced previously.
\begin{pro}\label{pro big}
	There exists a map \(\pst:\xt[KI]_0 \ra \cK(M)\) \st \(\pst\) has the property
	\begin{itemize}\vspace{-3ex}
		\item[\((P)_0\)] \(\pst(x)=\pht_{\ga}(x)\) for all \(x \in \xtk_0\).
	\end{itemize}\vspace{-3ex}
	Moreover, denoting 
	\[\Ps(v)=\Ps(T(v))=\pst(v)\cap \pst(T(v))\text{ for } v \in \xtki\; ;\] 
    \[\pst\circ\De^1_q =\xit^1_q,\quad \quad \pst\circ\De^2_q =\xit^2_q ,\]
	for \(q \in [Q]\) and \(m =\dim(e_q)=\dim(f_q)\), \(\pst|_{e_q[I]_0}\), \(\pst|_{f_q[I]_0}\) have the following properties.
	\begin{itemize}\vspace{-3ex}
		\item[$(P0)_{m,q}$]  For \(s=1,2\),
		\[\xit^s_q(\xi',iI^{-1})=\Om\left(\xit^s_q(\xi',(i-1)I^{-1}),\; \xit^s_q(\xi',1);\;B_i(q)\right)\]  for all \(\xi' \in \sci^{m-1}[I]_0 \setminus \left(\del \sci^{m-1}\right)[I]_0\) and \(i \in [I]\).
		\item[\((P1)_{m,q}\)] For all \(v \in \eqiz\), \(M=\pst(v) \cup \pst(T(v))\).
		\item[\((P2)_{m,q}\)] For all \(v \in e_q[I]_0\), \[\pst(v) \subset \bigcup_{x \in (e_q)_0}\pht_{\ga}(x).\]
		Similarly, for all \(v' \in f_q[I]_0\), 
		 \[\pst(v') \subset \bigcup_{x' \in (f_q)_0}\pht_{\ga}(x').\]
		\item[\((P3)_{m,q}\)] Let \[\mathscr{F}_q=\{s \in [Q]: e_s \prec e_q \text{ or } f_s \prec e_q\}.\]
		For all \(v \in e_q[I]_0\),
		\[\del\pst(v),\del \pst(T(v)) \subset \bigg\{\bigcup_{x \in \eqz} \Ph_{\ga}(x) \bigg\} \bigcup \bigg\{\bigcup_{i \in [I],\; s \in \mathscr{F}_q } \del B_i(s) \bigg\}.\]
		\item[\((P4)_{m,q}\)] For all \(v \in \eqiz\), \(\Ps(v)\setminus \cS=\mathbb{S}_1\;\dot\cup\;\mathbb{S}_2\;\dot\cup\;\dots\;\dot\cup\;\mathbb{S}_{l}\) where each \(\mathbb{S}_j\) is an open subset of a hypersurface belonging the set
		\[\Big\{\Ph_{\ga}(x):x \in \eqz \Big\} \bigcup \Big\{\del B_i(s):i \in [I],\; s \in \mathscr{F}_q \Big\}.\]
		Here \(\mathscr{F}_q\) is as in \((P3)_{m,q}\).
		\item[\((P5)_{m,q}\)] \sps \(v,v'\in \eqiz\) or \(v,v'\in \fqiz\); \(e\in e_q[I]_1\cup f_q[I]_1\) \st \(v,v'\prec e \). Then, there exists \(\cB=\bbb_{i(e)}\) for some \(i(e)\in [I]\) \st
		\begin{align*}
		\Ps(v)\cap \left(M\setminus(\cA \cup \cB)\right) & = \Ps(v')\cap \left(M\setminus(\cA \cup \cB)\right);\\
		\pst(v)\cap \left(M\setminus(\cA \cup \cB)\right) & = \pst(v')\cap \left(M\setminus(\cA \cup \cB)\right);\\
		\pst(T(v))\cap \left(M\setminus(\cA \cup \cB)\right) & = \pst(T(v')\cap \left(M\setminus(\cA \cup \cB)\right).
		\end{align*}
		Moreover, \(\cB\) can be characterized as follows. Without loss of generality, let \(v,v'\in \eqiz\). If \(D^1_q(v)=(\xi_1,\xi_2,\dots, \xi_m)\) and \(D^1_q(v')=(\xi'_1,\xi'_2,\dots, \xi'_m)\), there exists a unique index \(j \in [m]\) \st \(\xi_j \neq \xi'_j\). Let \(i \in [I]\) be such that \(\{\xi_j,\xi'_j\}=\{(i-1)I^{-1},iI^{-1}\}\). Then \(\cB=\bbb_i.\)
		\item[\((P6)_{m,q}\)] For all \(i \in [I]\) and \(v \in \eqiz\),
		\[\cH^n\left(\Ps(v)\cap \ov{\bbb^0_i}\right)<2^{4m+2}\et.\]
		\item[\((P7)_{m,q}\)] For all \(v \in \eqiz\),
		\[\cH^n\left(\Ps(v)\cap \ov{\cA}_1\right)<2^{4m+2}\et.\]
		\item[\((P8)_{m,q}\)] For all \(v \in \eqiz\), \(x \in (e_q)_0\) and \(R \in \mathscr{R}\),
		\[\Big|\cH^n\left(\Ps(v) \cap R\right)-\cH^n\left(\Ph_{\ga}(x) \cap R\right)\Big|<2^{4m+2}\et.\]
		In particular, 
		\begin{equation}
		\cH^n(\Ps(v))<L+(2^{4m+2}+2) \et.\label{H^n of Psi(v)}
		\end{equation}
	\end{itemize}\vspace{-2ex}
\end{pro} 
\begin{proof}
	We will inductively construct \(\pst:\eqiz, \fqiz \ra \cK(M)\). If \(\dim(e_q)=\dim(f_q)=0\), then \(e_q[I]=e_q\) and \(f_q[I]=f_q\). In that case we define \[\pst(e_q)=\pht_{\ga}(e_q) \quad \text{ and } \quad \pst(f_q)= \pht_{\ga}(f_q)\] which is precisely the property \((P)_0\). Moreover, for this definition of \(\pst\) on the \(0\)-cells \(e_q\),\(f_q\), properties \((P0)_{0,q}-(P8)_{0,q}\) are also satisfied. (We need to use Proposition \ref{pro existence of alpha}.)
	
	Let us assume that \(d\geq 1\) and \(\pst\) is defined on
	\[\bigcup_{\dim(e_q)\leq d-1}\Big(\eqiz \cup \fqiz\Big)\]
	and if \(\dim(e_s)=\dim(f_s)=(d-1)\), \(\pst|_{e_s[I]_0}\), \(\pst|_{f_s[I]_0}\) have the properties \((P0)_{(d-1),s}-(P8)_{(d-1),s}.\)
	\newcommand{\eniz}{e_{\nu}[I]_0}\newcommand{\fniz}{f_{\nu}[I]_0}
	
	Let \(\nu \in [Q]\) be \st \(\dim(e_{\nu})=\dim(f_{\nu})=d\). We define \(\xit^1_{\nu}, \;\xit^2_{\nu}\) on \(\sci^d[I]_0\) and \(\pst\) on \(e_{\nu}[I]_0,\; f_{\nu}[I]_0\) as follows. By our assumption, \(\pst\) is defined on \((\del e_{\nu})[I]_0,  \; (\del f_{\nu})[I]_0\). For \(z \in (\del \sci^d)[I]_0,\;  \De^1_{\nu}(z)\in (\del e_{\nu})[I]_0\) and \(\De^2_{\nu}(z) \in (\del f_{\nu})[I]_0\); for \(s=1,2\) we define
	\begin{equation}
	\xit^s_{\nu}(z)=\left(\pst \circ \De^s_{\nu}\right)(z). \label{eq defining xit-s on boundary}
	\end{equation}
	If \(z \in \sci ^d[I]_0 \setminus \left(\del \sci^d\right)[I]_0\), we write \(z=(z',iI^{-1})\) with \(z' \in \sci^{d-1}[I]_0\) and \(1 \leq i \leq 1-I^{-1}\); for \(s=1,2\), we define
	\begin{equation}
	\xit^s_{\nu}(z', iI^{-1}) = \Om \left(\xit^s_{\nu}(z',(i-1)I^{-1}),\; \xit^s_{\nu}(z',1);\;B_i(\nu) \right). \label{eq defining xit-s on the whole cube}
	\end{equation}
    For \(v \in e_{\nu}[I]_0 \cup f_{\nu}[I]_0\), we define
	\begin{equation}
		\pst(v)=
		\begin{cases}
		\left(\xit^1_{\nu}\circ D^1_{\nu}\right)(v) & \text{if } v \in e_{\nu}[I]_0; \\
		\left(\xit^2_{\nu}\circ D^2_{\nu}\right)(v) & \text{if } v \in f_{\nu}[I]_0. 
		\end{cases}
		\label{eq defining psi-tilde on eq, fq}
	\end{equation}
	We note that by \eqref{eq defining xit-s on boundary}, \(\pst\) is well-defined on \((\del e_{\nu})[I]_0\), \((\del f_{\nu})[I]_0\). Let us also set 
	\begin{equation}
		\Xi_{\nu}(z)=\xit^1_{\nu}(z) \cap \xit^2_{\nu}(z), \; z \in \sci^d[I]_0 \label{eq defining xi}
	\end{equation}
	so that 
	\begin{equation}
		\Ps(v)=
		\begin{cases}
		\left(\Xi_{\nu}\circ D^1_{\nu}\right)(v) & \text{if } v \in e_{\nu}[I]_0; \\
		\left(\Xi_{\nu}\circ D^2_{\nu}\right)(v) & \text{if } v \in f_{\nu}[I]_0. 
		\end{cases}\label{eq defining psi on eq, fq}
	\end{equation}
We need to show that \(\pst\) defined in this way on \(\eniz,\fniz\) have the properties \((P0)_{d,\nu}-(P8)_{d, \nu}\). 

Before we start, for convenience let us introduce the following notation. 
	\[B_i=B_i(\nu),\quad A_i=M\setminus\ov{B}_i, \quad S_i=\del B_i = \del A_i = \ov{A}_i \cap \ov{B}_i.\]
	We fix \(z' \in \sci^{d-1}[I]_0\setminus \left(\del \sci^{d-1}\right)[I]_0\) and define
	\[K_i = \xit^1_{\nu}(z', iI^{-1}), \quad L_i=\xit^2_{\nu}(z', iI^{-1}),\quad \Si_i=\Xi_{\nu}(z',iI^{-1})=K_i\cap L_i.\]	
	From \eqref{eq defining xit-s on the whole cube}, for \(i=1,2,\dots, (I-1)\) we obtain,
	\begin{align}
	& K_i = (K_{i-1}\cap \ov{A}_i) \cup (K_I \cap \ov{B}_i);\nonumber\\
	&L_i = (L_{i-1}\cap \ov{A}_i) \cup (L_I \cap \ov{B}_i); \label{eq inductive defn of K_i, L_i, Si_i} \\
	& \Si_i=(\Si_{i-1}\cap \ov{A}_i)\cup (\Si_I \cap \ov{B}_i) \cup (K_{i-1}\cap L_I\cap S_i) \cup (K_I\cap L_{i-1}\cap S_i). \nonumber
	\end{align}
Using induction	one can prove the following. 
\begin{align}
	& K_i=\Big[K_0\cap \left(\cup_{s=1}^i B_s\right)^c\Big] \bigcup \Big[K_I \cap \left(\cup_{s=1}^i \ov{B}_s\right)\Big]; \nonumber\\
	& L_i=\Big[L_0\cap \left(\cup_{s=1}^i B_s\right)^c\Big] \bigcup \Big[L_I \cap \left(\cup_{s=1}^i \ov{B}_s\right)\Big]; \label{eq explicit defn of K_i, L_i}\\
	& \Si_i=\Big[\Si_0\cap \left(\cup_{s=1}^i B_s\right)^c\Big] \bigcup \Big[\Si_I \cap \left(\cup_{s=1}^i \ov{B}_s\right)\Big] \bigcup \La_i \quad \text{where} \quad \La_i \subset \bigcup_{t=1}^i S_t.\nonumber
\end{align}
For \(i>1\) we have,
\begin{equation}
\La_i=(\La_{i-1}\cap \ov{A}_i) \cup (K_{i-1}\cap L_I \cap S_i) \cup (K_I\cap L_{i-1} \cap S_i). \label{eq reccurence for La_i}
\end{equation}
From the equations in \eqref{eq explicit defn of K_i, L_i}, we conclude that the equations in \eqref{eq inductive defn of K_i, L_i, Si_i} hold for \(i=I\) as well. We want to denote the top and bottom faces of \(e_{\nu}\) and \(f_{\nu}\) by the indices \(\tau\) and \(\be \in [Q]\) i.e.
\begin{align*}
	& \Big\{\De^1_{\nu}(\sci^{d-1} \times \{1\}),\;\De^2_{\nu}(\sci^{d-1} \times \{1\})\Big\}=\{e_{\ta}, f_{\ta}\}; \\
	& \Big\{\De^1_{\nu}(\sci^{d-1} \times \{0\}),\;\De^2_{\nu}(\sci^{d-1} \times \{0\})\Big\}=\{e_{\be}, f_{\be}\}.
\end{align*}
We note that the top face of \(e_{\nu}\) could be either \(e_{\ta}\) (in that case the top face of \(f_{\nu}\) is $f_{\tau}$) or \(f_{\tau}\) (in that case the top face of \(f_{\nu}\) is $e_{\tau}$); similar remark applies for the bottom face. For the ease of the presentation, we will assume that \(e_{\ta}\) is the top face of \(e_{\nu}\) and \(e_{\be}\) is the bottom face of \(e_{\nu}\). The other cases will be entirely analogous.

We now proceed in steps to check that \(\pst\vert_{\eniz}\), \(\pst\vert_{\fniz}\) have the properties \((P0)_{d,\nu}-(P8)_{d,\nu}\).

\textbf{Step 0.} \((P0)_{d, \nu}\) is satisfied because of our definition of \(\xit_{\nu}^1\) and \(\xit_{\nu}^2\) i.e. equation \eqref{eq defining xit-s on the whole cube}.

\textbf{Step 1.} By the induction hypothesis, \(\pst\) restricted to \(e_{\be}[I]_0, \;f_{\be}[I]_0\) (and \(e_{\ta}[I]_0,\; f_{\ta}[I]_0\)) satisfy \((P1)_{(d-1), \be}\) (and \((P1)_{(d-1), \ta}\)). Hence, 
\[K_0 \cup L_0=M=K_I \cup L_I\]
which together with \eqref{eq explicit defn of K_i, L_i} imply
\[M=K_i \cup L_i.\]

\textbf{Step 2.} By our assumption, \(\pst\) restricted to \(e_{\be}[I]_0, \;f_{\be}[I]_0\) (and \(e_{\ta}[I]_0,\; f_{\ta}[I]_0\)) satisfy \((P2)_{(d-1), \be}\) (and \((P2)_{(d-1), \ta}\)). Hence,
\[K_0 \subset \bigcup_{x \in (e_{\be})_0}\pht_{\ga}(x); \quad K_{I}\subset \bigcup_{x \in (e_{\ta})_0} \pht_{\ga}(x).\]
\tf using \eqref{eq explicit defn of K_i, L_i},
\[K_i \subset \bigcup_{x \in (e_{\nu})_0}\pht_{\ga}(x)\quad \forall i \in [I].\]
Similarly, we can show that
\[L_i \subset \bigcup_{x \in (f_{\nu})_0}\pht_{\ga}(x)\quad \forall i \in [I].\]

\textbf{Step 3.} By the induction hypothesis, \newcommand{\ebz}{(e_{\be})_0}\newcommand{\etz}{(e_{\ta})_0}\newcommand{\enz}{(e_{\nu})_0}
\begin{align}
	& \del K_0 \subset \bigg\{\bigcup_{x \in \ebz} \Ph_{\ga}(x) \bigg\} \bigcup \bigg\{\bigcup_{i \in [I],\; s \in \mathscr{F}_{\be} } \del B_i(s) \bigg\}; \label{eq del K_0}\\
	& \del K_I \subset\bigg\{\bigcup_{x \in \etz} \Ph_{\ga}(x) \bigg\} \bigcup \bigg\{\bigcup_{i \in [I],\; s \in \mathscr{F}_{\ta} } \del B_i(s) \bigg\}. \label{eq del K_I}
\end{align}
Moreover, by \eqref{eq boundary} and \eqref{eq inductive defn of K_i, L_i, Si_i}
\begin{align}
 \del K_i \subset \del K_{i-1} \cup \del K_{I} \cup S_i  \Longrightarrow \del K_i \subset \del K_0 \cup \del K_I \cup \left(\cup_{j=1}^i S_j\right). \label{eq del K_i}
\end{align}
Combining \eqref{eq del K_0}, \eqref{eq del K_I} and \eqref{eq del K_i}, we obtain
\begin{equation}
	\del K_i \subset\bigg\{\bigcup_{x \in \enz} \Ph_{\ga}(x) \bigg\} \bigcup \bigg\{\bigcup_{i \in [I],\; s \in \mathscr{F}_{\be} \cup \mathscr{F}_{\ta} } \del B_i(s) \bigg\}\bigcup \bigg\{\bigcup_{t=1}^{i} \del B_t\bigg\}.\label{eq del K_i final}
\end{equation}
We can arrive at a similar conclusion for \(\del L_i\) as well.

\textbf{Step 4.} For \(j = 0,1,\dots,I\), let \(\hat{\Si}_j:=\Si_j \setminus \cS\). By the induction hypothesis \[\hat{\Si}_0=\mathbb{S}^0_1\;\dot\cup\;\mathbb{S}^0_2\;\dot\cup\;\dots\;\dot\cup\;\mathbb{S}^0_{l_0}\]
where each \(\mathbb{S}^0_j\) is an open subset of a hypersurface belonging to the set 
\begin{equation*}
\Big\{\Ph_{\ga}(x):x \in \ebz \Big\} \bigcup \Big\{\del B_i(s):i \in [I],\; s \in \mathscr{F}_{\be}\Big\};
\end{equation*}
and 
\[\hat{\Si}_I=\mathbb{S}^I_1\;\dot\cup\;\mathbb{S}^I_2\;\dot\cup\;\dots\;\dot\cup\;\mathbb{S}^I_{l_I}\]
where each \(\mathbb{S}^I_j\) is an open subset of a hypersurface belonging to the set
\begin{equation}\label{eq induction hyp on Si_I-hat}
\Big\{\Ph_{\ga}(x):x \in \etz \Big\} \bigcup \Big\{\del B_i(s):i \in [I],\; s \in \mathscr{F}_{\ta} \Big\}. 
\end{equation}
By induction let us assume that 
\[\hat{\Si}_{i-1}=\mathbb{S}^{i-1}_1\;\dot\cup\;\mathbb{S}^{i-1}_2\;\dot\cup\;\dots\;\dot\cup\;\mathbb{S}^{i-1}_{l_{i-1}}\]
where each \(\mathbb{S}^{i-1}_j\) is an open subset of a hypersurface belonging to the set
\begin{equation}
	\Big\{\Ph_{\ga}(x):x \in \enz \Big\} \bigcup \Big\{\del B_i(s):i \in [I],\; s \in \mathscr{F}_{\be} \cup \mathscr{F}_{\ta} \Big\}\bigcup \Big\{\del B_t: t\in [i-1]\Big\}.\label{eq induction hypothesis for Si-hat-i}
\end{equation}
Then,
\begin{equation}
	(\Si_{i-1} \cap \ov{A}_i)\setminus \cS = (\hat{\Si}_{i-1} \cap \ov{A}_i)\setminus \cS = (\hat{\Si}_{i-1} \cap A_i)\setminus \cS \label{eq Si-hat cap A_i}
\end{equation}
as by \eqref{eq induction hypothesis for Si-hat-i} \(\hat{\Si}_{i-1}\cap \del B_i \subset \cS.\) Further, \((\hat{\Si}_{i-1} \cap A_i)\setminus \cS\) is an open subset of \(\hat{\Si}_{i-1}\) as \(A_i \subset M\) is open and \(\cS \subset M\) is closed. Similarly, using \eqref{eq induction hyp on Si_I-hat},
\begin{equation}
	(\Si_{I} \cap \ov{B}_i)\setminus \cS = (\hat{\Si}_{I} \cap B_i)\setminus \cS \label{eq Si_I cap A_i}
\end{equation}
which is an open subset of \(\hat{\Si}_{I}\).
\begin{equation}
	(K_{i-1} \cap L_I \cap S_i)\setminus \cS = (int(K_{i-1} \cap L_I )\cap S_i)\setminus \cS \label{eq S_i term one}
\end{equation}
as by \eqref{eq boundary}, \(\del(K_{i-1} \cap L_I) \subset \del K_{i-1} \cup \del L_I\); by \eqref{eq del K_i final} and by the induction hypothesis \((P3)_{(d-1), \ta}\),
\[(\del K_{i-1} \cup \del L_I)\cap S_i \subset \cS.\]
Similarly,
\begin{equation}
(K_{I} \cap L_{i-1} \cap S_i)\setminus \cS = (int(K_{I} \cap L_{i-1} )\cap S_i)\setminus \cS. \label{eq S_i term two}
\end{equation}
Moreover,
\[M = B_i \;\dot\cup\; S_i\; \dot\cup \;A_i.\]
Hence, using \eqref{eq inductive defn of K_i, L_i, Si_i} and equations \eqref{eq induction hyp on Si_I-hat} -- \eqref{eq S_i term two}, we obtain \[\hat{\Si}_{i}=\mathbb{S}^{i}_1\;\dot\cup\;\mathbb{S}^{i}_2\;\dot\cup\;\dots\;\dot\cup\;\mathbb{S}^{i}_{l_{i}}\]
 where each \(\mathbb{S}^{i}_j\) is an open subset of a hypersurface belonging to the set
 \begin{equation}
 \Big\{\Ph_{\ga}(x):x \in \enz \Big\} \bigcup \Big\{\del B_i(s):i \in [I],\; s \in \mathscr{F}_{\be} \cup \mathscr{F}_{\ta} \Big\}\bigcup \Big\{\del B_t: t\in [i]\Big\}.
 \end{equation}

\textbf{Step 5.}
\sps \(v, v'\in (\del e_{\nu})[I]_0\). Let \(\rh \in [Q]\) \st \(\la_{\rh} \prec e_{\nu}\) (\(\la_{\rh}\) could be \(e_{\rh}\) or \(f_{\rh}\)) with \(\dim(\la_{\rh})=d-1\) and \(v, v' \in \la_{\rh}[I]_0\). Then, by the inductive hypothesis \((P5)_{(d-1),\rh}\) and by the compatibility relation \eqref{eq compatibility}, we get \((P5)_{d, \nu}\) for this case.

Now, we assume that \(v \notin (\del e_{\nu})[I]_0\); \(D^1_q(v)=(z_1,z_2,\dots ,z_d)=(z',z_d)\), \(D^1_q(v')=(z'_1,z'_2,\dots ,z'_d)=(z'',z'_d)\). If \(z_d \neq z'_d\), using \eqref{eq inductive defn of K_i, L_i, Si_i} we get \((P5)_{d, \nu}\). Otherwise, \(z'_d=z_d\) and there are two possibilities: either \(v' \notin (\del e_{\nu})[I]_0\) or \(v' \in (\del e_{\nu})[I]_0\). Let us assume the second possibility (the other case is similar), i.e. there exists \(\rh \in [Q]\) \st \(\la_{\rh} \prec e_{\nu}\) (\(\la_{\rh}=e_{\rh} \text{ or } f_{\rh}\)) with \(\dim(\la_{\rh})=d-1\) and \(v' \in \la_{\rh}[I]_0\). We define, 
\[B'_i=B_i(\rh),\quad A'_i=M\setminus\ov{B'_i}, \quad S'_i=\del B'_i = \del A'_i ;\]
\[K'_i = \xit^1_{\nu}(z'', iI^{-1}), \quad L'_i=\xit^2_{\nu}(z'', iI^{-1}),\quad \Si'_i=\Xi_{\nu}(z'',iI^{-1})=K'_i\cap L'_i.\]
Using \((P0)_{(d-1),\rh}\) and the compatibility relation \eqref{eq compatibility}, we can deduce the following relations similar to those in \eqref{eq explicit defn of K_i, L_i}.
\begin{align}
& K'_i=\Big[K'_0\cap \left(\cup_{s=1}^i B'_s\right)^c\Big] \bigcup \Big[K'_I \cap \left(\cup_{s=1}^i \ov{B'_s}\right)\Big]; \nonumber\\
& L'_i=\Big[L'_0\cap \left(\cup_{s=1}^i B'_s\right)^c\Big] \bigcup \Big[L'_I \cap \left(\cup_{s=1}^i \ov{B'_s}\right)\Big]; \label{eq explicit defn of K'_i, L'_i}\\
& \Si'_i=\Big[\Si'_0\cap \left(\cup_{s=1}^i B'_s\right)^c\Big] \bigcup \Big[\Si'_I \cap \left(\cup_{s=1}^i \ov{B'_s}\right)\Big] \bigcup \La'_i \quad \text{where} \quad \La'_i \subset \bigcup_{t=1}^i S'_t.\nonumber
\end{align}
\((P5)_{(d-1),\be}\) and \((P5)_{(d-1),\ta}\) imply
\begin{align}
	& \Si_0\cap \left(M\setminus(\cA \cup \cB)\right)  = \Si'_0\cap \left(M\setminus(\cA \cup \cB)\right); \nonumber\\
	& \Si_I\cap \left(M\setminus(\cA \cup \cB)\right) = \Si_I'\cap \left(M\setminus(\cA \cup \cB)\right); \label{P5 Si_I}
\end{align}
where \(\cB=\bbb_j\); \(j\) is such that if \(z_l \neq z'_l\), \(\{z_l,z'_l\}=\{(j-1)I^{-1}, jI^{-1}\}\). We note that for all \(s \in [I]\),
\begin{equation*}
B_s \;\De\; B'_s,\; \del B_s,\; \del B'_s,\; \subset \cA.
\end{equation*}
Hence,
\begin{align}
 \left(\cup_{s=1}^i B_s\right)^c \De \left(\cup_{s=1}^i B'_s\right)^c \subset \cA \implies &\left(\cup_{s=1}^i B_s\right)^c \cap (\cA \cup \cB)^c = \left(\cup_{s=1}^i B'_s\right)^c \cap (\cA \cup \cB)^c;\nonumber \\
 \left(\cup_{s=1}^i \ov{B_s}\right) \De \left(\cup_{s=1}^i \ov{B'_s}\right) \subset \cA \implies & \left(\cup_{s=1}^i \ov{B_s}\right)\cap (\cA \cup \cB)^c = \left(\cup_{s=1}^i \ov{B'_s}\right)\cap (\cA \cup \cB)^c;\nonumber\\
 \La_i \cap (\cA \cup \cB)^c &= \emptyset = \La_i'\cap (\cA \cup \cB)^c \label{P5 intersection w. A^c}
\end{align}
Using the equations \eqref{eq explicit defn of K_i, L_i}, \eqref{eq explicit defn of K'_i, L'_i}, \eqref{P5 Si_I} and \eqref{P5 intersection w. A^c}, we obtain 
\[\Si_i\cap \left(M\setminus(\cA \cup \cB)\right)  = \Si'_i\cap \left(M\setminus(\cA \cup \cB)\right).\]
for all \(i \in [I]\). By a similar argument, one can also show that 
\[K_i\cap \left(M\setminus(\cA \cup \cB)\right)  = K'_i\cap \left(M\setminus(\cA \cup \cB)\right);\]
and 
\[L_i\cap \left(M\setminus(\cA \cup \cB)\right)  = L'_i\cap \left(M\setminus(\cA \cup \cB)\right).\]

\textbf{Step 6.} Using \eqref{eq reccurence for La_i}, for \(i>1\),
\begin{equation}
	\cH^n(\La_i) \leq \cH^n(\La_{i-1})+\cH^n(K_{i-1}\cap L_{I} \cap S_i) + \cH^n(K_I\cap L_{i-1} \cap S_i). \label{eq H^n of La_i}
\end{equation}
By \((P2)_{d,\nu}\),
\begin{equation*}
	K_{i-1} \subset \bigcup_{x \in \enz} \pht_{\ga}(x), \quad L_I \subset \bigcup_{x \in (f_{\nu})_0} \pht_{\ga}(x).
\end{equation*}
Hence,\newcommand{\phtg}{\pht_{\ga}}\newcommand{\fnz}{(f_{\nu})_0}
\begin{equation}
	\cH^n(S_i \cap K_{i-1} \cap L_I) \leq \sum_{\substack{x_1 \in \enz,\\ x_2 \in (f_{\nu})_0}} \cH^n\left(S_i \cap \pht_{\ga}(x_1) \cap \phtg(x_2)\right). \label{eq H^n Si cap dots 1}
\end{equation}
For \(x_1 \in \enz,\; x_2 \in \fnz,\)
\begin{align}
	\cH^n\left(S_i \cap \pht_{\ga}(x_1) \cap \phtg(x_2)\right) &= \cH^n\left(S_i \cap \pht_{\ga}(x_1) \cap \phtg(x_2) \cap \phtg(T(x_2))\right) \nonumber\\
	& + \cH^n\left(S_i \cap \pht_{\ga}(x_1) \cap \phtg(x_2) \cap \left(M -\phtg(T(x_2))\right)\right) \nonumber \\
	& = 0 + \cH^n\left(S_i \cap \left(\pht_{\ga}(x_1)-\phtg(T(x_2)\right)\right) \nonumber \\
	& <\frac{2^{2d}\et}{I}. \label{eq H^n Si cap dots 2}
\end{align}
In the second equality we have used the fact that \(S_i\) is transverse to \(\Ph_{\ga}(x_2)=\phtg(x_2) \cap \phtg(T(x_2))\) (which was assumed in the definition of \(r_i(q)\)) and \((P1)_{d,\nu}\). In the last inequality we have used Proposition \ref{pro existence of alpha}, item (v). \tf combining \eqref{eq H^n of La_i} -- \eqref{eq H^n Si cap dots 2}, we obtain
\begin{equation*}
\cH^n(\La_i) < \cH^n(\La_{i-1})+\frac{2^{4d+1}\et}{I}\; \text{ for } i>1.
\end{equation*}
Moreover, by the above argument,
\begin{equation*}
	\cH^n(\La_1) < \frac{2^{4d+1}\et}{I}.
\end{equation*}
Hence, for all \(i \in [I],\)
\begin{equation}
\cH^n(\La_i) < \frac{2^{4d+1}\et i}{I}\leq 2^{4d+1}\et. \label{eq H^n La_i final}
\end{equation}
Using this inequality along with \eqref{eq explicit defn of K_i, L_i} and \((P6)_{(d-1),\be},\;(P6)_{(d-1),\ta}\), we conclude that for all \(j \in [I],\) 
\begin{align}
	\cH^n\left(\Si_i \cap \ov{\bbb_j^0}\right) & \leq \cH^n\left(\Si_0 \cap \ov{\bbb_j^0}\right) + \cH^n\left(\Si_I \cap \ov{\bbb_j^0}\right) + \cH^n(\La_i) \nonumber \\
	& <(2^{4d-2}+2^{4d-2}+2^{4d+1})\et\nonumber\\
	& < 2^{4d+2}\et. \label{eq proving P6}
\end{align}

\textbf{Step 7.} Using \eqref{eq H^n La_i final}, \eqref{eq explicit defn of K_i, L_i} and \((P7)_{(d-1),\be},\;(P7)_{(d-1),\ta}\), we obtain as in Step 6,
\begin{align}
\cH^n\left(\Si_i \cap \ov{\cA}_1\right) & \leq \cH^n\left(\Si_0 \cap \ov{\cA}_1\right) + \cH^n\left(\Si_I \cap \ov{\cA}_1\right) + \cH^n(\La_i) \nonumber \\
& <(2^{4d-2}+2^{4d-2}+2^{4d+1})\et \nonumber\\
& < 2^{4d+2}\et. \label{eq proving P7}
\end{align}

\textbf{Step 8.} We note that for all \(s \in [I]\),
\[\ov{\bbb}_s^c \subset B_s^c \subset \ov{\bbb}_s^c \cup \ov{\cA}; \qquad \ov{B}_s \subset \bbb_s \subset \ov{B}_s \cup \ov{\cA}.\]
\tf for all \(i \in [I]\),
\begin{align*}
& \Big(\bigcup_{s=1}^i\ov{\bbb}_s\Big)^c \subset \Big(\bigcup_{s=1}^iB_s\Big)^c \subset \Big(\bigcup_{s=1}^i\ov{\bbb}_s\Big)^c \cup \ov{\cA};\\
& \bigcup_{s=1}^i\ov{B}_s \subset \bigcup_{s=1}^i\bbb_s \subset \Big(\bigcup_{s=1}^i\ov{B}_s\Big) \cup \ov{\cA}.
\end{align*}
Hence, using \eqref{eq explicit defn of K_i, L_i}, we deduce the following inequalities.
\begin{align}
&\cH^n(\Si_i \cap R) \leq \cH^n\left(\Si_0\cap \left(\cup_{s=1}^i\ov{\bbb}_s\right)^c\cap R\right)+\cH^n\left(\Si_0 \cap \ov{\cA}\right)+\cH^n\left(\Si_I\cap \left(\cup_{s=1}^i\bbb_s\right)\cap R\right) + \cH^n(\La_i);\label{inequ 1}\\
& \cH^n(\Si_i \cap R) \geq \cH^n\left(\Si_0\cap \left(\cup_{s=1}^i\ov{\bbb}_s\right)^c\cap R\right)+\cH^n\left(\Si_I\cap \left(\cup_{s=1}^i\bbb_s\right)\cap R\right)-\cH^n\left(\Si_I \cap \ov{\cA}\right).\label{inequ 2}
\end{align}
Denoting
\begin{equation}
	R_1 = \left(\cup_{s=1}^i\ov{\bbb}_s\right)^c\cap R, \qquad R_2 = \left(\cup_{s=1}^i\bbb_s\right)\cap R, \label{eq defn of R1, R2}
\end{equation}
using \eqref{inequ 1}, \eqref{inequ 2}, \((P7)_{(d-1),\be}\), \((P7)_{(d-1),\ta}\) and \eqref{eq H^n La_i final} we obtain,
\begin{align}
	&\left|\cH^n(\Si_i \cap R)-\cH^n(\Si_0 \cap R_1)-\cH^n(\Si_I \cap R_2)\right|\nonumber\\
	& \leq \cH^n\left(\Si_0 \cap \ov{\cA}\right) + \cH^n\left(\Si_I \cap \ov{\cA}\right) + \cH^n(\La_i) \nonumber \\
	& \leq (2.2^{4d-2}+2^{4d+1})\et.\label{inequ 3}
\end{align}
Let \(x_0 \in \enz\); without loss of generality we can assume that \(x_0 \in \ebz\) i.e. \(D^1_{\nu}(x_0)=(\xi, 0)\) for some \(\xi \in \sci^{d-1}\). Let \(x_1\in\etz \) \st \(x_1=\De^1_{\nu}(\xi,1).\)
We note that \(R \in \mathscr{R}\) implies \(R_1, R_2 \in \mathscr{R}\) as well. Further, by the definition of \(\bbb_s\), \(\norm{\Ph_{\ga}(x_0)}(\del \bbb_s)=0\); hence,
\[\cH^n\left(\Ph_{\ga}(x_0)\cap R\right)=\cH^n\left(\Ph_{\ga}(x_0)\cap R_1\right) + \cH^n\left(\Ph_{\ga}(x_0)\cap R_2\right).\]
\tf using \(\eqref{inequ 3}\), \((P8)_{(d-1),\be}\), \((P8)_{(d-1), \ta}\), Proposition \ref{pro existence of alpha} (i) and \eqref{cond: estimating mass norm} we get the following estimate.
\begin{align}
	&\left|\cH^n(\Si_i \cap R)-\cH^n\left(\Ph_{\ga}(x_0) \cap R\right)\right| \nonumber\\
	& \leq \left|\cH^n(\Si_i \cap R)-\cH^n(\Si_0 \cap R_1)-\cH^n(\Si_I \cap R_2)\right| + \left|\cH^n(\Si_0 \cap R_1)-\cH^n\left(\Ph_{\ga}(x_0) \cap R_1\right)\right| \nonumber\\
	& + \left|\cH^n(\Si_I \cap R_2)-\cH^n\left(\Ph_{\ga}(x_1) \cap R_2\right)\right| + \left|\cH^n\left(\Ph_{\ga}(x_1) \cap R_2\right)-\cH^n\left(\Ph_{\ga}(x_0) \cap R_2\right)\right|\nonumber\\
	& <(2^{4d-1}+2^{4d+1}+2^{4d-2}+2^{4d-2}+3)\et\nonumber\\
	& <2^{4d+2}\et. \label{eq proving P8}
\end{align}
Since \(M \in \mathscr{R}\), \eqref{eq proving P8} and Proposition \ref{pro existence of alpha}, (ii) imply
\[\cH^n(\Si_i)\leq L+(2^{4d+2}+2)\et.\]
\end{proof}
\subsection{Approximate solution of the Allen-Cahn equation}
Here we will briefly discuss an approximate solution of the Allen-Cahn equation whose energy is concentrated in a tubular neighbourhood of a closed, two-sided hypersurface (with mild singularities). We will follow the paper by Guaraco \cite{Guaraco}; further details can be found there.

Let \(h: \bbr \ra \bbr\) be the unique solution of the following ODE.
\[\dot{\vp}(t)= \sqrt{2W(\vp(t))};\quad \vp(0)=0.\]
For all \(t \in \bbr\), \(-1<h(t)<1\) and as \(t \ra \pm \infty\), \((h(t)\mp 1)\) converges to zero exponentially fast.

\(h_{\ve}(t)=h(t/\ve)\) is a solution of the one dimensional Allen-Cahn equation
\[\ve^2\ddot{\vp}(t)=W'(\vp(t))\]
with finite total energy:
\[\int_{-\infty}^{\infty}\left(\frac{\ve}{2}\dot{h}_{\ve}(t)^2+W(h_{\ve}(t))\right) dt=2\si.\]
For $\ve >0$, we define Lipschitz continuous function \newcommand{\he}{h_{\ve}}\newcommand{\se}{\sqrt{\ve}}\newcommand{\ue}{u_{\ve}}
\begin{equation*}
g_{\ve}(t)=
\begin{cases}
h_{\ve}(t) & \text{if } |t|\leq \sqrt{\ve};\\
h_{\ve}(\se)+\left(\frac{t}{\se}-1\right)(1-\he(\se)) &\text{if } \se \leq t \leq 2 \se;\\
1 &\text{if } t\geq 2\se;\\
\he(-\se)+ \left(\frac{t}{\se}+1\right)(1+\he(-\se)) &\text{if } -2\se \leq t \leq -\se;\\
-1 &\text{if } t\leq -2\se.
\end{cases}
\end{equation*}
Suppose \(d:M \ra \bbr\) is a Lipschitz continuous function \st \(\norm{\nabla d}=1\) a.e. Let \(\ue=g_{\ve}\circ d\), \(U \subset M\). Using the notation
\[E_{\ve}(u,U)=\int_U\ve \frac{|\nabla u|^2}{2}+\frac{W(u)}{\ve},\]
we compute
\begin{align}
E_{\ve}(u_{\ve},U) & = \int_{U \cap \{|d|\leq 2\se\}} \left[\frac{\ve}{2}\dot{g}_{\ve}(d(y))^2+\frac{1}{\ve}W\Big(g_{\ve}(d(y))\Big)\right] \; d\cH^{n+1}(y) \nonumber \\
& = \int_{|\ta|\leq 2\se} \left[\frac{\ve}{2}\dot{g}_{\ve}(\ta)^2+\frac{1}{\ve}W(g_{\ve}(\ta))\right] \cH^n \left(U \cap \{d=\ta\}\right) d\ta \nonumber \\
& \leq \int_{|\ta|\leq \se} \left[\frac{\ve}{2}\dot{h}_{\ve}(\ta)^2+\frac{1}{\ve}W(h_{\ve}(\ta))\right] \cH^n \left(U \cap \{d=\ta\}\right) d\ta \label{estimate energy of u_ep1} \\
& + \left(\frac{1}{2}(1-\he(\se))^2+\frac{1}{\ve}W(\he(\se))\right)\vol(M,g) \label{estimate energy of u_ep2}
\end{align}
Let 
\[\La_{\ve} = \sup_{|\ta|\leq \se} \cH^n\left(U \cap \{d = \ta\}\right).\]
Then, the integral in \eqref{estimate energy of u_ep1} is bounded by \(2 \si \La_{\ve}\). Further, there exists \(\ve_0=\ve_0(W,g,\et)\) \st for all \(\ve \leq \ve_0\) the expression in \eqref{estimate energy of u_ep2} is bounded by  \(2\si \et\). Hence,
\begin{equation}
	\label{estimate energy of u_ep final}
	E_{\ve}(u_{\ve},U) \leq 2\si(\La_{\ve} + \et)\; \forall \ve \leq \ve_0.
\end{equation}
For \(v \in \xtki\), we define \(d_v:M \ra \bbr\) as follows.
\begin{equation}
d_v(p)=
\begin{cases}
 -d(p,\Ps(v)) &\text{if } p \in \pst(v);\\
 d(p,\Ps(v)) &\text{if } p \in \pst(T(v)).
\end{cases}
\label{def d_v}
\end{equation}
By the definition of \(\Ps(v)\) and \((P1)\) of Proposition \ref{pro big}, \(d_v\) is a well-defined continuous function on \(M\). Moreover, by \cite{Guaraco}*{Proposition 9.1}, \(d_v\) is Lipschitz continuous and \(\norm{\nabla d_v}=1\) a.e.

Let \(p \in U \cap \{d_v = \tau\}\). Then, either \(d(p,\cS)=\md{\ta}\) or there exists \(z \in (\Ps(v)\setminus \cS)\) with \(d(p,z)=\md{\ta}\) \st \(p = \exp_z(\ta \mathbf{n}(z))\) where \(\mathbf{n}(z)\) is the unit normal to \(\Ps(v)\setminus \cS\) at \(z\) pointing inside \(\pst(T(v))\). We must have
\[z \in \cN_{\md{\ta}}(U)\cap \left(\Ps(v)\setminus \cS\right).\]
Hence, we can write
\begin{equation}
\cH^n\left(U\cap \{d_v=\ta\}\right) \leq \cH^n(\cT_{\md{\ta}}(\cS))+\int_{\cN_{\md{\ta}}(U)\cap (\Ps(v)\setminus \cS)}\left|J \exp_{z}(\ta \mathbf{n}(z))\right|\; d\cH^n(z) 
\label{inequ H^n of d_v=tau}
\end{equation}
where \(J \exp_{z}(\ta \mathbf{n}(z))\) is the Jacobian factor of the map \(z \mapsto \exp_{z}(\ta \mathbf{n}(z))\) which can be estimated as follows. 

Let
\[\mathscr{S}'=\big\{\Ph_{\ga}(x): x \in \tx[K]_0\big\} \bigcup \left\{\del B_{i}(q):i \in [I], q \in [Q]\right\}.\]
Using \((P4)\) of Proposition \ref{pro big} and \cite{warner}*{Corollary 4.2, Theorem 4.3}, \cite{Guaraco}*{Proposition 9.4}, one can deduce the following estimate. Let \(\la>0\) be \st
\[ I\!I_{\Si}(v,v)\leq \la \langle v,v \rangle \; \forall\; \Si \in \mathscr{S}', \; v \in T\Si.\]
Here \(I\!I_{\Si}\) denotes the second fundamental form of \(\Si.\) Then, there exist \(\ta_0\) and \(C_1\) depending only on \(\la\), the ambient dimension \(n+1\) and \(g\) \st
\[\left|J \exp_{z}(\ta \mathbf{n}(z))\right|\leq (1+C_1|\ta|) \; \forall z \in \Ps(v)\setminus \cS.\]
Hence, using \eqref{eq area of equidistant hypersurface} and \eqref{inequ H^n of d_v=tau}, there exists \(\ta_1=\ta_1(\mathscr{S},\mathscr{S}',n,g)\) \st for all \(\md{\ta}\leq \ta_1\)
\[\cH^n\left(U\cap \{d_v=\ta\}\right) \leq C|\ta|+(1+C_1|\ta|)\cH^n\left(\cN_{\md{\ta}}(U) \cap \Ps(v) \right).\]
This estimate along with \((P6)\) and \((P7)\) of Proposition \ref{pro big}, \eqref{H^n of Psi(v)} and \eqref{estimate energy of u_ep final} gives the following proposition. We recall that \(k\) is the dimension of the parameter spaces \(X\) and \(\tx\).
\begin{pro}\label{pro est energy of g_ep circ d_v}
	There exists \(\ve_1=\ve_1(\et,\ta_1,\ve_0,\de,r_0-r_1)\) \st for all \(v \in \xtki\), \(\ve \leq \ve_1\) and \(i \in [I]\), denoting \(\vartheta_{\ve}^v=g_{\ve}\circ d_v\), we have
	\[E_{\ve}\left(\vartheta_{\ve}^v,\bbb_i^1\right) \leq 2\si (2^{4k+2}+2)\et; \; E_{\ve}\left(\vartheta_{\ve}^v,\cA_2\right) \leq 2\si(2^{4k+2}+2)\et; \; E_{\ve}\left(\vartheta_{\ve}^v,M\right) \leq 2\si(L+ (2^{4k+2}+4)\et).\] 	
\end{pro}\vspace{-3ex}
We recall from Section 3.2 that \(f:M \ra [1/3,2/3]\) is a Morse function with no non-global local maxima or minima; hence the map \(t \mapsto f^{-1}(t)\) is continuous in the Hausdorff topology. Following \cite{Guaraco}*{Section 7}, we define the following functions. For \(t \in [1/3,2/3]\), let
\[d^{(t)}(p)=
\begin{cases}
-d\left(p, f^{-1}(t)\right) & \text{if } f(p)\leq t;\\
d\left(p,f^{-1}(t)\right) & \text{if } f(p) \geq t.
\end{cases}\]
We define \(w_{\ve}:[0,1] \ra H^1(M)\) as follows.
\begin{equation}w_{\ve}(t)=
\begin{cases}
g_{\ve}\circ d^{(t)} & \text{if } \frac{1}{3}\leq t \leq \frac{2}{3}; \\
1-3t(1-w_{\ve}(1/3)) & \text{if } 0 \leq t \leq \frac{1}{3};\\
-1+3(1-t)(1+w_{\ve}(2/3)) & \text{if } \frac{2}{3}\leq t\leq 1.
\end{cases}\label{w_ep}
\end{equation}
Since \(f^{-1}(t)\) varies continuously in the Hausdorff topology, \(w_{\ve}:[0,1] \ra H^1(M)\) is continuous \cite{Guaraco}*{Proposition 9.2}. The following proposition follows from \cite{Guaraco}*{Section 9.6} and the estimates \eqref{eq: mass of modified morse fn level set}, \eqref{estimate energy of u_ep final}.
\begin{pro}\label{pro est energy of w_t}
	There exists \(\ve_2=\ve_2(\et,f,\ve_0,\de,r_0-r_1)\) \st for all \(t \in [0,1]\), \(\ve \leq \ve_2\) and \(i \in [I]\), we have
\[E_{\ve}\left(w_{\ve}(t),\bbb_i^1\right) \leq 6\si\et; \quad E_{\ve}\left(w_{\ve}(t),\cA_2\right) \leq 6\si\et.\]	
\end{pro}\vspace{-4ex}
\subsection{Construction of sweepout in \(H^1(M)\setminus\{0\}\) and proof of Theorem \ref{thm main thm}}
The next proposition essentially follows from the argument in \cite{gt}*{Proof of Lemma 7.5}.
\begin{pro}\label{mod u}
	\(u \mapsto |u|\) is a continuous map from \(H^1(M)\) to itself. 
\end{pro}\vspace{-4ex}
\begin{proof}
	We define \( \ka:\bbr \ra \bbr\) as follows
	\[\ka(t)=
	\begin{cases}
	1 &\text{if } t>0;\\
	0 &\text{if } t=0;\\
	-1&\text{if } t<0.
	\end{cases}\]
By \cite{gt}*{Lemma 7.6} \(v \in H^1(M)\) implies \(|v|\in H^1(M)\) with \(D|v|=\ka(v)Dv\).

We need to show that if \(u_n \ra u\) in \(H^1(M)\), \(|u_n| \ra |u|\) in \(H^1(M)\). We will use the following fact. Let \(\{x_n\}\) be a sequence in a topological space \(\cX\) and \(x \in \cX\) \st every subsequence \(\{x_{n_k}\}\) of \(\{x_n\}\) has a further subsequence \(\{x_{n_{k_i}}\}\) which converges to \(x\). Then \(\{x_n\}\) converges to \(x\). Therefore, without loss of generality we can assume that \(u_n\) converges to \(u\) pointwise a.e.

\(|u_n|\ra |u|\) in \(L^2(M)\) because of the inequality \(||u_n|-|u||\leq |u_n-u|\). Using the fact that \(Du=0\) a.e. on \(\{u=0\}\), we compute
\begin{align}
\int_M |\ka(u_n)Du_{n}-\ka(u)Du|^2 & \leq \int_M 2|\ka(u_n)|^2|Du_n-Du|^2+2|\ka(u_n)-\ka(u)|^2|Du|^2 \nonumber\\
& \leq 2\int_M|Du_n-Du|^2 +2\int_{\{u \neq 0\}} |\ka(u_n)-\ka(u)|^2|Du|^2. \nonumber
\end{align}
The first integral converges to \(0\). If \(u(y) \neq 0\) and \(u_n(y)\) converges to \(u(y)\), \(\ka(u_n(y))\) converges to \(\ka(u(y))\) as well. Hence, by the dominated convergence theorem, the second integral also converges to \(0\).
\end{proof}
For \(x \in \bbr\), let \(x^{+}=\max\{x,0\}\); \(x^{-}=\min\{x,0\}\). We define the maps \(\ph,\ps,\tht :H^1(M)\times H^1(M)\times H^1(M) \ra H^1(M)\) as follows.
\begin{align}
	&\ph(u_0,u_1,w)=\min\{\max\{u_0,-w\},\max\{u_1,w\}\}; \nonumber\\
	&\ps(u_0,u_1,w)=\max\{\min\{u_0,w\},\min\{u_1,-w\}\}; \label{def of phi, psi}\\
	&\tht(u_0,u_1,w)=\ph(u_0,u_1,w)^{+}+\ps(u_0,u_1,w)^{-}.\nonumber
\end{align}
\begin{pro}
	The map \(\tht\) defined above has the following properties.
\begin{itemize}\vspace{-3ex}\label{pro tht}
	\item[(i)] \(\tht\) is a continuous map.
	\item[(ii)] \(\tht(-u_0,-u_1,w)=-\tht(u_0,u_1,w)\).
	\item[(iii)] If \(|u_0|,|u_1| \leq 1\), then \(\tht(u_0,u_1,\mathbf{1})=u_0\), \(\tht(u_0,u_1,-\mathbf{1})=u_1\) where \(\mathbf{1}\) denotes the constant function \(1\).
	\item[(iv)] If \(u_0(p)=u_1(p)\), then \(\tht(u_0,u_1,w)(p)=u_0(p)=u_1(p)\).
	\item[(v)] For every \(U \subset M\), \(E_{\ve}(\tht(u_0,u_1,w),U)\leq E_{\ve}(u_0,U)+E_{\ve}(u_1,U)+E_{\ve}(w,U)\).
	\item[(vi)] \(\tht(u_0,u_1,w)(p)=0\) implies \(p\in \{u_0=0\}\cup \{u_1=0\}\cup\{w=0\}\).
\end{itemize}\vspace{-2ex}
\vspace{-2ex}
\end{pro}
\begin{proof}
	We have the following identities.
	\[\max\{a,b\}=\frac{a+b+|a-b|}{2}; \quad \min\{a,b\}=\frac{a+b-|a-b|}{2}.\]
	Hence, the continuity of \(\ph,\;\ps\) and \(\tht\) follows from Proposition \ref{mod u}. We note that \(\ph(-u_0,-u_1,w)=-\ps(u_0,u_1,w)\) which implies (ii). (iii) also follows from direct computation. Indeed, if \(|u_0|,|u_1| \leq 1\),  \(\ph(u_0,u_1,\mathbf{1})=u_0=\ps(u_0,u_1,\mathbf{1})\) and \(\ph(u_0,u_1,-\mathbf{1})=u_1=\ps(u_0,u_1,-\mathbf{1})\). To prove (iv), we need the following identity. If \(a\geq 0\),
\begin{equation}
	\min\{\max\{a,-b\},\max\{a,b\}\}=a. \label{id}
\end{equation} 
Indeed, the identity holds if \(\max\{a,b\}=a=\max\{a,-b\}\). It also holds if \(\max\{a,-b\}=a\) and \(\max\{a,b\}=b>a\). It is not possible for both \(\max\{a,-b\}\) and \(\max\{a,b\}\) to be different from \(a\) as this will imply \(b>a,\; -b>a\; \implies 0>a\). Moreover, if \(a\geq 0\), either \(\min\{a,b\}\geq 0\) (if \(b \geq 0\)) or \(\min\{a,-b\}\geq 0\) (if \(b \leq 0\)); hence
\begin{equation}
	\max\{\min\{a,b\},\min\{a,-b\}\}\geq 0. \label{id'}
\end{equation}
\eqref{id} and \eqref{id'} give (iv) if \(u_0(p)=u_1(p)\geq 0\). The other case is similar.

For \(i=0,1\), let \(U_i=\{u_i<0\}\) and \(V_i=\{u_i>0\}\). We also set \(G=\{w<0\}\), \(H=\{w>0\}\), \(Z=\{w=0\}\). Then we have (suppressing the dependence of \(\ph,\;\ps,\;\tht\) on \(u_0,\;u_1,\;w\)),
\begin{align}
	&\ph(y) <0 \Leftrightarrow y \in (U_0\cap H)\dot\cup(U_1\cap G);\; \ph(y)>0 \Leftrightarrow y \in (V_0\cap H)\dot\cup (V_1\cap G)\dot\cup(V_0\cap V_1 \cap Z);\nonumber\\
	&\ps(y)<0 \Leftrightarrow y \in (U_0\cap H)\dot\cup (U_1\cap G)\dot\cup(U_0\cap U_1 \cap Z); \; \ps(y)>0 \Leftrightarrow y \in (V_0\cap H)\dot\cup (V_1\cap G). \label{U_i}
\end{align}
Further,
\begin{equation}
	\tht(y)=
	\begin{cases}
	\ph(y) &\text{if } y \in \{\ph\geq 0\}\cap \{\ps\geq 0\};\\
	\ps(y) &\text{if } y \in \{\ph\leq 0\}\cap \{\ps\leq 0\}.
	\end{cases}\label{theta}
\end{equation}
It follows from \eqref{U_i} that
\[M= \big(\{\ph\geq 0\}\cap \{\ps\geq 0\}\big) \cup \big(\{\ph\leq 0\}\cap \{\ps\leq 0\}\big).\]
Hence, using the definition of \(\ph\) and \(\ps\), \eqref{theta} implies that
\begin{equation}
\tht(y) \in \{u_0(y),u_1(y),w(y),-w(y)\};\quad \nabla\tht(y) \in \{\nabla u_0(y),\nabla u_1(y), \nabla w(y), -\nabla w(y)\}. \label{pt w}
\end{equation}
\tf we have the following point-wise estimates.
\[|\nabla \tht (y)|^2 \leq |\nabla u_0 (y)|^2+|\nabla u_1 (y)|^2+|\nabla w (y)|^2;\; W(\tht(y))\leq W(u_0(y))+W(u_1(y))+W(w(y))\] 
which give item (v) of the proposition. \eqref{pt w} also implies item (vi).
\end{proof}\newcommand{\xtkin}{\tilde{X}[KI]}\newcommand{\vr}{\varrho}

Let \(\al \in \xtkin_j\) so that \(\mu=\pi(\al)\in X[KI]_j\). We define the following maps.
\begin{equation*}
	\left(\pi|_{\al}\right)^{-1}\circ \De_{\mu}=\De^{\bullet}_{\al}:\sci^j\ra \al;\quad D_{\mu}\circ\left(\pi|_{\al}\right)=D^{\bullet}_{\al}:\al \ra \sci^j.
\end{equation*}\newcommand{\zee}{\ze_{\ve}}\newcommand{\zeh}{\hat{\ze}_{\ve}}
\begin{pro}\label{pro zeta_ep}
	Let \(\ve^*=\min\{\ve_1,\ve_2,(\de/4)^2, 4^{-1}(r_1-r)^2\}\). There exists a continuous, \(\bbz_2\)-equivariant map \(\ze_{\ve}:\tx \ra H^1(M)\setminus \{0\}\) with the property
	\begin{itemize}\vspace{-3ex}
		\item[\((\cP)_0\)] For all \(v \in \xtki\), \(\zee(v)=\vartheta_{\ve}^v\);
	\end{itemize}\vspace{-2ex}
where \(\vartheta_{\ve}^v\) is as in Proposition \ref{pro est energy of g_ep circ d_v}. Moreover, if \(\al \in \xtkin_j\), \(\zee|_{\al}\) has the following properties.
    \begin{itemize}\vspace{-3ex}
	\item[\((\cP0)_{j,\al}\)] Let \(\zeh^{\al}=\zee\circ\De^{\bullet}_{\al}\). Then, for all \(z=(z',z_j)\in \sci^j\) (with \(z'\in \sci^{j-1}\)),
	\[\zeh^{\al}(z',z_j)=\tht\left(\zeh^{\al}(z',0),\zeh^{\al}(z',1),w_{\ve}(z_j)\right);\]
	where \(w_{\ve}\) is as defined in \eqref{w_ep}.
	\item[\((\cP1)_{j,\al}\)] We define \(\cI_{\al}=\{i(e):e \in \al_1\}\) where \(i(e)\) is as in Proposition \ref{pro big}, \((P5)\). For all \(\ve\leq \ve^*\), \(t \in \al\) and \(v \in \al_0\),
	\[y \notin \cA_2\bigcup\Big(\bigcup_{i \in \cI_{\al}}\bbb^1_{i}\Big)  \implies \zee(t)(y)=\zee(v)(y).\]
	\item[\((\cP2)_{j,\al}\)] For all \(\ve\leq \ve^*\), \(i\in [I]\) and \(t \in \al\),
	\[\frac{1}{2\si }E_{\ve}\left(\zee(t),\bbb^1_i\right)\leq 3^j(2^{4k+2}+2)\et;\quad \frac{1}{2\si }E_{\ve}\left(\zee(t),\cA_2\right)\leq 3^j(2^{4k+2}+2)\et.\]
As a consequence, for all \(\ve\leq \ve^*\) and \(t \in \al\),
\begin{equation}
	\frac{1}{2\si }E_{\ve}\left(\zee(t),M\right)\leq L+(2^{4k+2}+4)\et +(j+1)3^j(2^{4k+2}+2)\et. \label{eqn of pro 3.13}
\end{equation}
    \end{itemize}\vspace{-2ex}
\end{pro}
\begin{proof}
	The map \(\zee\) will be defined inductively on the cells of \(\xtkin\). The equivariance of \(\zee\) and the fact that \(\zee(x)\) is not identically \(0\) for all \(x \in \tx\) will be proved at the end (after proving the other properties of \(\zee\)). If \(v \in \xtki\), we define \(\zee(v)=\vartheta_{\ve}^v\) which is precisely the property \((\cP)_0\). The properties \((\cP0)_{0,v}\) and \((\cP1)_{0,v}\) are vacuous; \((\cP2)_{0,v}\) is satisfied because of Proposition \ref{pro est energy of g_ep circ d_v}. Let us assume that for some \(p\geq 1\), \(\zee\) has been defined on
	\[\bigcup_{j\leq p-1}\xtkin_j\]
	and if \(\vr \in \xtkin_{p-1}\), \(\zee\vert_{\vr}\) has the properties \((\cP0)_{(p-1),\vr}-(\cP2)_{(p-1),\vr}.\)
	
	Let \(\la \in \xtkin_p.\) By our assumption, \(\zee\) is already defined on \(\del \la.\) For \(z \in \del \sci^p\), we define
	\[\zeh^{\la}(z)=\left(\zee \circ \De^{\bullet}_{\la}\right)(z).\]
	For arbitrary \(z=(z',z_p) \in \sci^p\) with \(z'\in \sci^{p-1}\), we define
\begin{equation}
	\zeh^{\la}(z',z_p)=\tht\left(\zeh^{\la}(z',0),\zeh^{\la}(z',1), w_{\ve}(z_p)\right) \label{defn zeh}
\end{equation}
 where \(w_{\ve}\) is as defined in \eqref{w_ep}. By Proposition \ref{pro tht} (iii), this equation indeed holds for \(z_p=0,1\). Further, \eqref{defn zeh} is also valid if \(z \in \del \sci^p\) because of our induction hypothesis that \(\zee\) restricted \(\del \la\) satisfies \((\cP0)\) and the compatibility relation \eqref{eq compatibility}. For \(t \in \la\), \(\zee(t)\)	is defined by \(\zee(t)=\left(\zeh^{\la} \circ D^{\bullet}_{\la}\right)(t)\). \(\zee\) is continuous on \(\la\) as \(\zee|_{\del \la}\) is \cts and the maps \(\tht,\;w_{\ve}\) are continuous.
 
It remains to verify that \(\zee|_{\la}\) satisfies the properties \((\cP0)_{p,\la}-(\cP2)_{p,\la}\). \((\cP0)_{p,\la}\) holds by the definition \eqref{defn zeh}. To prove \((\cP1)_{p,\la}\), we first show that if \(v,v' \in \xtki\) and \(e \in \xtkin_1\) \st \(v,v' \prec e\),
\begin{equation}
	y \notin \cA_2 \cup \bbb_{i(e)}^1 \implies \zee(v)(y)=\zee(v')(y). \label{cP2 1st}
\end{equation}
We set \(\cB^1=\bbb_{i(e)}^1\) and \(\cB=\bbb_{i(e)}\). Either \(y \in \pst(v)\) or \(y \in \pst(T(v))\). We assume that \(y \in \pst(T(v))\) (the other case is analogous); hence, (using Proposition \ref{pro big} \((P5)\))
\[y \in \pst(T(v))\cap \left(\cA_2 \cup \cB^1\right)^c = \pst(T(v'))\cap \left(\cA_2 \cup \cB^1\right)^c \]
Thus, \(d_v(y)=d(y,\Ps(v))\), \(d_{v'}(y)=d(y,\Ps(v'))\). If both \(d_v(y)\) and \(d_{v'}(y)\) are \(\geq 2\sqrt{\ve}\), \(\zee(v)(y)=1=\zee(v')(y)\). 

Otherwise, let us say \(d_v(y)=d(y,\Ps(v))<2\sqrt{\ve}\). We will show that \(d_v(y)=d_{v'}(y).\) Let \(y'\in \Ps(v)\) \st \(d(y,y')= d(y,\Ps(v))\). \tf
\begin{equation}
	d(y,y')= d(y,\Ps(v)) <2\sqrt{\ve}\leq \min\{\de/2, r_1-r\}. \label{d(y,y')}
\end{equation}
Hence, \(y \notin \cA_2 \cup \cB^1\) implies \(y' \notin \cA \cup \cB\). Thus, (again using Proposition \ref{pro big} \((P5)\))
\[y'\in \Ps(v)\cap \left(\cA \cup \cB\right)^c =  \Ps(v')\cap \left(\cA \cup \cB\right)^c.\]
This gives 
\[d(y, \Ps(v))=d\left(y,\Ps(v)\cap \left(\cA \cup \cB\right)^c\right);\]
and \(d(y, \Ps(v'))\leq d(y,y')< 2\sqrt{\ve}\) which also implies (by the above reasoning)
\[d(y, \Ps(v'))=d\left(y,\Ps(v')\cap \left(\cA \cup \cB\right)^c\right).\]
Since \(\Ps(v)\cap \left(\cA \cup \cB\right)^c = \Ps(v')\cap \left(\cA \cup \cB\right)^c\), we get \(d_v(y)=d_{v'}(y).\)

Thus we have proved \eqref{cP2 1st} which together with \eqref{defn zeh}, induction hypothesis \((\cP1)_{(p-1),\del \la}\) and Proposition \ref{pro tht} (iv) gives \((\cP1)_{p,\la}\). 

Finally we prove \((\cP2)_{p,\la}\). Using \eqref{defn zeh} together with Proposition \ref{pro tht} (v), the induction hypothesis \((\cP2)_{(p-1),\del \la}\) and Proposition \ref{pro est energy of w_t} we obtain the following estimate (\(\cU\) stands for \(\bbb_i^1\) for some \(i \in [I]\) or \(\cA_2\).)
\begin{align*}
\frac{1}{2\si}E_{\ve}\left(\zeh^{\la}(z',z_p), \cU\right)& \leq \frac{1}{2\si}E_{\ve}\left(\zeh^{\la}(z',0),\cU\right)+ \frac{1}{2\si}E_{\ve}\left(\zeh^{\la}(z',1), \cU\right) + \frac{1}{2\si}E_{\ve}\left(w_{\ve}(z_p), \cU\right)\\
&\leq 2.3^{p-1}(2^{4k+2}+2)\et + 3\et\\
&\leq 3^{p}(2^{4k+2}+2)\et.
\end{align*}
This estimate along with \((\cP1)_{p,\la}\), \((\cP)_0\) and Proposition \ref{pro est energy of g_ep circ d_v} gives for all \(\ve\leq \ve^*\), \(t \in \al\) and \(v \in \al_0\),
\begin{align*}
\frac{1}{2\si}E_{\ve}\left(\ze_{\ve}(t),M\right) & \leq \frac{1}{2\si}E_{\ve}\left(\ze_{\ve}(v),M\right)+\sum_{i \in \cI_{\la}}\frac{1}{2\si}E_{\ve}\left(\ze_{\ve}(t),\bbb_i^1\right) +\frac{1}{2\si}E_{\ve}\left(\ze_{\ve}(t),\cA_2\right)\\
&\leq L+(2^{4k+2}+4)\et +(p+1)3^p(2^{4k+2}+2)\et.
\end{align*}
In the last line we have used the fact that \(\md{\cI_{\la}}\leq p\) which follows from the characterization of \(i(e)\) in Proposition \ref{pro big} \((P5).\)

Now we show that the map \(\zee\) constructed above is \(\bbz_2\)-equivariant. From the definition of \(d_v\) (\eqref{def d_v}), it follows that \(d_{T(v)}=-d_v\) for all \(v \in \xtki.\) Hence, by \((\cP)_0\), \(\zee(T(v))=-\zee(v)\) for all \(v \in \xtki.\) Now we can use induction along with \((\cP0)\) and Proposition \ref{pro tht} (ii) to conclude that \(\zee(T(x))=-\zee(x)\) for all \(x \in \xt.\)

Lastly, we prove that for all \(x \in \tx\), \(\cH^{n+1}(\{\zee(x)=0\})=0\). This will in particular imply that for all \(x \in \tx\), \(\zee(x)\) is not identically equal to \(0\). If \(v \in \xtki\), by \eqref{def d_v} and \((\cP)_0\), \(\{\zee(v)=0\}=\Ps(v)\) and \(\cH^{n+1}(\Ps(v))=0\) by Proposition \ref{pro big} $(P4)$. Hence, \(\cH^{n+1}(\{\zee(v)=0\})=0\) for all \(v \in \xtki\). Now, as before we can use induction along with \((\cP0)\) and Proposition \ref{pro tht} (vi) to conclude that \(\cH^{n+1}(\{\zee(x)=0\})=0\) for all \(x \in \xt.\)
\end{proof}
\begin{proof}[Proof of Theorem \ref{thm main thm}]
	By Proposition \ref{pro zeta_ep}, for every \(\et>0\) there exists \(\ve^*>0\) \st for all \(\ve\leq \ve^*\) there exists \(\ze_{\ve}\in \tilde{\Pi}\) \st
	\[\frac{1}{2\si}\sup_{x \in \tx}E_{\ve}(\ze_{\ve}(x)) \leq L+\al(k)\et \;(\text{by } \eqref{eqn of pro 3.13})\]
	where \(\al(k)\) is a constant which depends only on \(k=\dim (\tx)\). Hence,
\[\frac{1}{2\si}\bL_{\ve}(\tilde{\Pi})\leq L+\al(k)\et \;\; \forall\ve \leq \ve^* \implies\frac{1}{2\si}\limsup_{\ve \ra 0^{+}} \bL_{\ve}(\tilde{\Pi})\leq L+\al(k)\et.\]
Since \(\et>0\) is arbitrary, this implies \eqref{width ineq}.

The following facts follow from \cite{GG1}*{Section 6}. \(\Ind_{\bbz_2}(\xt)\geq p+1\) if and only if each map \(\Ph \in \Pi \) is a \(p\)-sweepout. Moreover,
\[c_{\ve}(p)=\inf_{\tilde{\Pi}} \bL_{\ve}(\tilde{\Pi})\]
where the infimum is taken over all \(\tilde{\Pi}\) \st \(\Ind_{\bbz_2}(\xt)\geq p+1\). Similarly,
\[\om_p = \inf_{\Pi}\bL_{AP}(\Pi)\]
where the infimum is taken over all \(\Pi\) \st each map in the homotopy class \(\Pi\) is a \(p\)-sweepout. 

We fix \(p\in \bbn\). For each \(j \in \bbn\), there exist double cover \(\tx_j \ra X_j\) and the corresponding homotopy classes \(\Pi_j\), \(\tilde{\Pi}_j\) (as discussed in Section 1) \st \(\Ind_{\bbz_2}(\xt_j)\geq p+1\) and
\[\bL_{AP}(\Pi_j)<\om_p+\frac{1}{j}.\]
By \eqref{width ineq},
\[\frac{1}{2\si}\limsup_{\ve \ra 0^{+}}\bL_{\ve}(\tilde{\Pi}_j)\leq \bL_{AP}(\Pi_j)<\om_p+\frac{1}{j}.\]
Hence, there exists \(\tilde{\ve}>0\) \st for all \(\ve\leq \tilde{\ve}\)
\[\frac{1}{2\si}c_{\ve}(p)\leq \frac{1}{2\si}\bL_{\ve}(\tilde{\Pi}_j)<\om_p+\frac{1}{j}\]
which implies 
\[\frac{1}{2\si}\limsup_{\ve \ra 0^{+}}c_{\ve}(p)\leq \om_p+\frac{1}{j}.\]
Since this holds for all \(j \in \bbn\), we obtain \eqref{spec ineq}.
\end{proof}
\section{Proof of Theorem \ref{thm critical set}}
In this section we will discuss how Theorem \ref{thm critical set} can be proved using the results contained in the papers \cite{HT,Guaraco,GG1}. We recall the function \(F\) defined in \eqref{2 defn F}. For \(\be\in (0,1)\), we set \(\si_{\be}=F^{-1}(1-\be).\)

\begin{pro}\label{pro 4.1}
	Let \(\{u_{i}:M\ra (-1,1)\}_{i=1}^{\infty}\) be a sequence smooth functions \st the followings hold.
	\begin{itemize}\vspace{-3ex}
		\item[(i)] \(AC_{\ve_i}(u_i)=0\) with \(\ve_i\ra 0\) as \(i \ra \infty\).
		\item[(ii)] There exists \(E_0\) \st \(E_{\ve_i}(u_i)\leq E_0\) for all \(i \in \bbn\).
		\item[(iii)] \(V\) is a stationary, integral varifold \st \(V_i=V[u_i]\ra V\) and \(\spt(V)\) has optimal regularity.
	\end{itemize}\vspace{-3ex}
Then, for all \(s>0\) there exists \(b_0\in (0,1/2]\) \st the following holds. Denoting \(w_i=F \circ u_i\), for all \(b \in (0,b_0]\) there exists \(i_0 \in \bbn\) depending on \(s\) and \(b\) \st for all \(i \geq i_0\) there exists \(t_i \in [-\si_b/2,\si_b/2]\) for which \(\{w_i>t_i\},\{w_i<t_i\}\in \cm\), \(\db{\{w_i>t_i\}}+\db{\{w_i<t_i\}}=\db{M}\), \(\del\db{\{w_i>t_i\}}=\db{\{w_i=t_i\}} \) and \(\bF(V,\md{\{w_i=t_i\}})\leq s.\)
\end{pro}\vspace{-4ex}
\begin{proof}
	Throughout the proof, we will use the notation of \cite{HT}. Let \(p\in (M,g)\) and \(0<r<\text{inj}(M)/4\). Identifying \(\bbr^{n+1}\) with \(T_pM\), let us define \(f_{p,r}:B^{n+1}(\mathbf{0},4)\ra M\),
	\[f_{p,r}(v)=\exp_{p}(rv)\quad \text{ and }\quad g_{p,r}=r^{-2}f_{p,r}^{*}g.\]
	 As explained in \cite{Guaraco}, there exists \(0<r_0<\text{inj}(M)/4\), depending on \((M^{n+1},g)\), \st if \(p \in M\), \(0<r\leq r_0\) and \(\{u_i\}_{i=1}^{\infty}\) are solutions of \(AC_{\ve_i}(u_i)=0\) on \((B^{n+1}(\mathbf{0},4),g_{p,r})\) with \(\ve_i\ra 0\), all the results of \cite{HT} continue to hold.
	 
	 For simplicity, let us assume that \(V=N|\Si|\) where \(\Si\) is a closed, minimal hypersurface with optimal regularity. The general case will be similar. Let 
	 \begin{equation}
	 	s_0=2\si(N+4+2\si N)^{-1}. \label{4 defn of s_0}
	 \end{equation}
	 and we fix \(0<s\leq s_0\). (The assumption \(s\leq s_0\) will be useful later.) By \cite{HT}*{Proposition 5.1} there exists \(b_0 \in (0,1/2]\) \st
	 \begin{equation*}
	 	\limsup_{i \ra \infty}\int_{\{|u_i|\geq 1-b_0\}}\frac{W(u_i)}{\ve_i} \leq s. 
	 \end{equation*}\newcommand{\ssubset}{\subset\joinrel\subset}
Let us fix \(b \in (0,b_0]\). Using the above equation and \cite{HT}*{Proposition 4.3},
	 \begin{equation}
	 \label{4 small energy}
	 \limsup_{i \ra \infty}\int_{\{|u_i|\geq 1-b\}}|\nabla w_i| \leq s.
	 \end{equation}
	 
	 There exists an open set \(\Om\) containing \(reg(\Si)\) \st \(d_{\Si}=d(-,\Si)\) is smooth on \(\Om\) and the nearest point projection map \(P:\Om \ra \Si\) is also smooth on \(\Om\). By \cite{MN_JDG_survey}*{Proof of Theorem 5.2}, \(im(P)=reg(\Si)\). We choose \(\cU_1 \ssubset \cU_2 \ssubset reg(\Si)\) \st 
	 \begin{equation}
	 	N\cH^n(\Si \setminus\cU_1)<s.\label{4 U_0^c has small area}
	 \end{equation}
	  For \(\cU \ssubset reg(\Si)\), let
	 \[\mathfrak{N}_{r}\cU= \{v: v\in T^{\perp}_p\Si \text{ with } \norm{v}<r,  p \in \cU\}.\]
There exists \(\rh>0\) \st \(\exp:\mathfrak{N}_{2\rh}\cU_2 \ra \Om\) is a diffeomorphism onto its image. For \(j=1,2\), let \(U_j=\exp(\mathfrak{N}_{\rh}\cU_j)\) so that
\[P(\ov{U}_j)=\ov{\cU}_j=\Si \cap \ov{U}_j \quad \text{ and }\quad U_1 \ssubset U_2 \ssubset \Om.\]
We can assume that (choosing a smaller \(\rh\) if necessary), the Jacobian factor
\begin{equation}
	JP(y,S)\leq 2 \; \forall\; (y,S)\in G_n\ov{U}_2.\label{4 JP is bdd}
\end{equation}

On \(\Om\), following \cite{HT}*{Section 5}, we define
\begin{equation*}
	v_i=\begin{cases}
	\frac{\langle\nabla u_i,\nabla d_{\Si}\rangle}{|\nabla u_i|} & \text{if } |\nabla u_i|\neq 0;\\
	0 & \text{if } |\nabla u_i|=0.
	\end{cases}
\end{equation*}
Let \(\ta_i=(1-v_i^2)\ve_i|\nabla u_i|^2\). We choose a compactly supported function \(0\leq \ch_1\leq 1\) \st \(\text{spt}(\ch_1)\subset \Om\) and \(\ch_1 \equiv 1\) on \(\ov{U}_2.\) Since, \(V_i\ra V\) is the varifold sense, using \cite{HT}*{Proposition 4.3} we obtain (\cite{HT}*{(5.7)})
\begin{equation}
	\lim_{i\ra \infty}\int_{\Om}\ch_1\ta_i=\lim_{i\ra \infty}\int_{\Om}\ch_1(1-v_i^2)|\nabla w_i|=0\implies \lim_{i\ra \infty}\int_{\ov{U}_2}\ta_i=0 .\label{4 tilt converges to 0}
\end{equation}
Let us use the notation
\[\xi_i=\frac{\ve_i|\nabla u_i|^2}{2}-\frac{W(u_i)}{\ve_i}.\]
Since the level sets  \(u_i^{-1}(t)\) converge to \(\Si\) in the Hausdorff sense, we can choose a sequence \(r_i\) \st
\begin{itemize}\vspace{-3ex}
	\item \(r_i \ra 0\);
	\item \(\{|u_i|\leq 1-b\}\subset \cN_{r_i}(\Si);\)
	\item \(\ve_i/r_i \ra 0.\)
\end{itemize}\vspace{-3ex}
By \cite{HT}*{Proposition 4.3} and \eqref{4 tilt converges to 0},
\[\lim_{i \ra \infty}\int_{\ov{U}_2}|\xi_i|+\ta_i =0.\]
Let us choose a sequence \(\et_i \ra 0\) \st 
\[\lim_{i \ra \infty}\et_i^{-1}\int_{\ov{U}_2}|\xi_i|+\ta_i =0.\]
There exists \(i_1 \in \bbn\) \st for all \(i \geq i_1\) we have
\begin{itemize}\vth
	\item \(r_i \leq r_0 \) where \(r_0\) is as defined at the beginning of the proof;
	\item \(r_i<d(\ov{U}_1,\del U_2)\).
\end{itemize}\vth\newcommand{\scgi}{\mathscr{G}_i}
As we discussed at the beginning of the proof, for any \(p \in M\) and \(0<r\leq r_0\), if \(u_{\ve}\) is a solution of \(AC_{\ve}(u_{\ve})=0\) on \((B^{n+1}(\mathbf{0},4), g_{p,r})\), Proposition 5.5 and 5.6 of \cite{HT} hold (with the constants depending additionally on \(r_0\)). For our fixed choice of \(s\) and \(b\), we choose \(L\) via Proposition 5.5 and 5.6. For \(i \geq i_1\), let
\[G_i=\ov{U}_1\cap \{y:\int_{B(y,r)}|\xi_i|+\ta_i \leq \et_ir^n \text{ if }4\ve_iL\leq r \leq r_i\}.\]
\(G_i\) may not be \(\cH^n\)-measurable; for our later purpose we choose an \(\cH^n\)-measurable set \(\mathscr{G}_i \subset G_i\) as follows. By the Besicovitch covering theorem, there exist \(\{\mathscr{B}_j\}_{j=1}^{\ell}\) \st each \(\mathscr{B}_j\) is a collection of at-most countably many mutually disjoint open balls and
\[\ov{U}_1\setminus G_i \subset \bigcup_{j=1}^{\ell}\bigcup_{B \in \mathscr{B}_j} B.\]
Let
\[\scgi = \left( \bigcup_{j=1}^{\ell}\bigcup_{B \in \mathscr{B}_j}B \right)^c \cap \ov{U}_1\subset G_i \]
which is a compact set. Using the monotonicity formula for the scaled energy one can show the following (\cite{HT}*{(5.9)}).
\begin{equation}
	\norm{V_i}\left(\ov{U}_1\setminus \mathscr{G}_i\right) + \cH^{n}\left(P(\ov{U}_1\setminus \mathscr{G}_i)\right)\leq c(s,W,g)\et_i^{-1}\int_{\ov{U}_2}|\xi_i|+\ta_i \label{4 G_i}
\end{equation}
which converges to \(0\) as \(i \ra \infty\) by our choice of the sequence \(\{\et_i\}\). \newcommand{\sgi}{\mathscr{G}_i}

We define
\[S_i^t=\{w_i=t\}.\]
For a.e. $t\in [-\si_b/2,\si_b/2]$, \(\{w_i>t\},\{w_i<t\}\in \cm\); \(\db{\{w_i>t\}}+\db{\{w_i<t\}}=\db{M}\); and \(\del\db{\{w_i>t\}}=\db{\{w_i=t\}} \). For such \(t\in [-\si_b/2,\si_b/2]\), \(i\geq i_1\) and Lipschitz continuous function \(\vp:G_nM^{n+1}\ra \bbr\) with \(|\vp|\leq 1\), \(\text{Lip}(\vp)\leq 1\), we obtain the following estimates.\newcommand{\sit}{S^t_i}
\begin{align}
&\Md{|S^t_i|(\vp)-N\md{\Si}(\vp)}\leq \Md{\md{S^t_i \cap \ov{U}_1}(\vp)-N\md{\Si \cap \ov{U}_1}(\vp)} + \norm{\sit}(U_1^c) + N\norm{\Si}(U_1^c); \label{4.1}
\end{align}\vfo
\begin{align}
\Md{\md{S^t_i \cap \ov{U}_1}(\vp)-N\md{\Si \cap \ov{U}_1}(\vp)} \leq \Md{&\md{S^t_i \cap \sgi}(\vp)-N\md{P(\sgi)}(\vp)}\nonumber\\
&+\norm{\sit}(\ov{U}_1\setminus \sgi)+ N \cH^n(P(\ov{U}_1\setminus\sgi)); \label{4.2}
\end{align}\vfo
\begin{align}
	\Md{\md{S^t_i \cap \sgi}(\vp)-N\md{P(\sgi)}(\vp)}& \leq \Md{\md{S^t_i \cap \sgi}(\vp)-P_{\#}\md{S^t_i \cap \sgi}(\vp)}\nonumber\\
	& +\Md{P_{\#}\md{S^t_i \cap \sgi}(\vp)-N\md{P(\sgi)}(\vp)}; \label{4.3}
\end{align}\vfo
\begin{align}
	& \Md{\md{S^t_i \cap \sgi}(\vp)-P_{\#}\md{S^t_i \cap \sgi}(\vp)}\nonumber\\
	& \leq \int_{S_i^t\cap\; \sgi}\Md{\vp(y,T_y\sit)-\vp\left(P(y),DP|_y(T_y\sit)\right)JP(y,T_y\sit)}\; d\cH^n(y).\label{4.4}
\end{align}
Using \cite{HT}*{Proposition 5.6} and a scaling argument as in the Proof of Theorem 1 of \cite{HT}, as \(i \ra \infty\)
\begin{equation*}
d_{G_nM}\left(\left(y,T_y\sit\right),\left(P(y), DP|_y(T_y\sit)\right)\right) \ra 0 \text{ and }JP(y,T_y\sit) \ra 1
\end{equation*}
uniformly for \(|t|\leq \si_b/2\) and \(y\in S_i^t\cap \sgi\). Here \(d_{G_nM}\) denotes the distance in \(G_nM\). Hence, there exists a sequence of positive real numbers \(\{\tht_i\}_{i=1}^{\infty}\) (which does not depend on \(t\)) \st 
\begin{equation}
\lim_{i \ra \infty} \tht_i = 0\;  \text{ and }\; 	\sup_{\vp}\Md{\md{S^t_i \cap \sgi}(\vp)-P_{\#}\md{S^t_i \cap \sgi}(\vp)} \leq \tht_i \cH^n(\sit) \label{4.5}
\end{equation}
for all \(t\in [-\si_b/2,\si_b/2].\) 

Let us choose a cut-off function \(0\leq \ch'_2\leq 1\) with \(\spt(\ch'_2)\subset U_2\) and \(\ch'_2 \equiv 1\) on \(\ov{U}_1\). If \(\mathbf{p}:G_nM\ra M\) is the canonical projection map, \(\ch_2=\ch'_2\circ \mathbf{p}.\)

  For \(y \in P(\sgi)\), let \(m_i^t(y)\) be the cardinality of the set $P^{-1}(y)\cap\sit\cap\sgi$. Using \cite{HT}*{Proposition 5.5, 5.6} and a scaling argument as in the Proof of Theorem 1 of \cite{HT}, there exists \(i_2\in \bbn\) \st for all \(i \geq i_2\), \(y \in P(\sgi)\) and \(|t|\leq \si_b/2\), \(m_i^t(y)\leq N\). (Here we need to use \(s\leq s_0\).) Hence,
\begin{align}
& \Md{P_{\#}\md{S^t_i \cap \sgi}(\vp)-N\md{P(\sgi)}(\vp)}\nonumber\\
&\leq \int_{P(\sgi)}(N-m_i^t(y))\; d\cH^n(y)\nonumber\\
&= N\md{P(\sgi)}(\ch_2)-P_{\#}\md{\sit \cap \sgi}(\ch_2)\nonumber\\
&\leq N\md{\Si}(\ch_2)-P_{\#}\md{\sit}(\ch_2)+2\norm{\sit}(\ov{U}_1\setminus\sgi)+2\norm{\sit}(U_1^c)\;(\text{by }\eqref{4 JP is bdd}).\label{4.6}
\end{align}

We will now integrate the various error terms obtained above with respect to \(t\) in the interval \([-\si_b/2,\si_b/2]\) and let \(i \ra \infty\).
\begin{align}
	&\limsup_{i \ra \infty}\int_{-\si_b/2}^{\si_b/2}\norm{\sit}(U_1^c)\;dt\leq \limsup_{i \ra \infty}\si\norm{V_i}(U_1^c)\leq \si s \;(\text{by } \eqref{4 U_0^c has small area});\\	
	&\int_{-\si_b/2}^{\si_b/2}N\norm{\Si}(U_1^c)\; dt\leq \si s \;(\text{by } \eqref{4 U_0^c has small area});\\
	&\lim_{i \ra \infty}\int_{-\si_b/2}^{\si_b/2}\norm{\sit}(\ov{U}_1\setminus \sgi)+ N \cH^n(P(\ov{U}_1\setminus\sgi))\;dt =0 \; (\text{by \eqref{4 G_i}});\\
	&\limsup_{i \ra \infty}\int_{-\si_b/2}^{\si_b/2}\norm{\sit}(M) \;dt \leq \limsup_{i \ra \infty}\si\norm{V_i}(M)=\si N \cH^n(\Si).
\end{align}
We note that \(V_i \ra N|\Si|\) in the varifold sense implies \(P_{\#}V_i(\ch_2)\ra N|\Si|(\ch_2).\) (\(P_{\#}V_i(\ch_2)\) is well-defined as \(\spt(\ch_2)\subset U_2.\))  Hence, using \eqref{4 JP is bdd} and \eqref{4 small energy},
\begin{align}
&\limsup_{i \ra \infty}\int_{-\si_b/2}^{\si_b/2}\left(N\md{\Si}(\ch_2)-P_{\#}\md{\sit}(\ch_2)\right) dt \nonumber\\
&\leq \lim_{i \ra \infty}\si \left(N\md{\Si}(\ch_2)-P_{\#}V_i(\ch_2)\right)+\limsup_{i \ra \infty}2\int_{\{|u_i|\geq 1-b\}}|\nabla w_i|\nonumber\\
&\leq 2s. \label{4.7}
\end{align}

Hence, using the equations \eqref{4.1}-\eqref{4.7}, we conclude that for all \(i \geq i_2\), there exists a measurable function \(\Th_i:[-\si_b/2,\si_b/2]\ra \bbr\) \st 
\[\bF(N|\Si|,\md{\sit})\leq \Th_i(t)\;\forall t\in [-\si_b/2, \si_b/2] \; \text{ and }\; \limsup_{i \ra \infty}\int_{-\si_b/2}^{\si_b/2}\Th_i(t) \; dt\leq (2+4\si)s.\]
Hence, there exists \(i_3 \in \bbn\) \st for all \(i\geq i_3\),
\[\int_{-\si_b/2}^{\si_b/2}\Th_i(t) \; dt\leq (3+4\si)s.\]
Therefore, for \(i \geq i_3\), there exists \(t_i \in [-\si_b/2, \si_b/2] \) \st
\[\bF\left(N|\Si|,\md{S_i^{t_i}}\right)\leq \Th_i(t_i)\leq (3+4\si)\si_b^{-1}s.\]
Since \(s\in (0,s_0]\) is arbitrary and \(b\leq b_0\leq 1/2\), this finishes the proof of the proposition.
\end{proof}

\begin{pro}\label{pro 4.2}
	Let \(\{u_{i}:M \ra (-1,1) \}_{i=1}^{\infty}\) be a sequence of smooth functions \st items (i) -- (iii) of Proposition \ref{pro 4.1} are satisfied. Additionally, we assume that \(u_i\) is a min-max critical point of \(E_{\ve_i}\) corresponding to the homotopy class \(\tilde{\Pi}\). Let us set \(L_{\ve_i}=\bL_{\ve_i}(\tilde{\Pi})\) so that
	\[L=\bL_{AP}(\Pi)=\frac{1}{2\si} \lim_{i \ra \infty}L_{\ve_i}=\norm{V}(M).\]
	 Then, (using the notation from Proposition \ref{pro 4.1}) for every \(s>0\) and \(b \in (0,b_0(s)]\), there exists \(i_0^{*}\geq i_0\) \st the following holds. For all \(i \geq i_0^{*}\), there exists \(\Ph_i:X \ra \cZ_n(M^{n+1};\bM;\bbz_2)\), \(x_i^{*}\in X\) and \(\de_i > 0\) \st \(\Ph_i \in \Pi\), \(\de_i \ra 0\),
	\begin{equation}
		\sup_{x \in X}\bM\left(\Ph_i(x)\right)\leq \max \left\{\frac{1}{2\si_b}(L_{\ve_i}+\ve_i),L+s\right\} + \de_i \quad \text{ and } \quad \bF\left(V,\md{\Ph_i(x_i^{*})}\right)\leq s. \label{eqn of prop 4.2}
	\end{equation}
\end{pro}\vfo
\begin{proof}
	Since \(u_i\) is a min-max critical point of \(E_{\ve_i}\) corresponding to the homotopy class \(\tilde{\Pi}\), for each \(i\), there exists a sequence of continuous, \(\bbz_2\)-equivariant maps \(\{h^i_j:\xt \ra H^1(M)\setminus\{0\}\}_{j=1}^{\infty}\) \st
	\[\sup_{x \in \xt,\; j \in \bbn}E_{\ve_i}\left(h^i_j(x)\right)\leq L_{\ve_i}+\ve_i\quad \text{ and }\quad \lim_{j\ra \infty}d_{H^1(M)}\left(u_i,h^i_j(\xt)\right)=0. \]
 The next Lemma is a restatement of Lemma 8.10 and 8.11 of \cite{Guaraco}.
\begin{lem}[\cite{Guaraco}*{Lemma 8.10, 8.11}]
	Let \(h_1,h_2\in H^1(M)\). For \(\de\in (0,1)\), we set \(C_{\de}=W(1-\de)>0.\) Then, for all \(\ve>0\), \[\cH^{n+1}\left(\{\md{h_1}\leq 1-\de\}\right)\leq \ve C_{\de}^{-1}E_{\ve}(h_1).\] 
	Let \(\al_0\in (-1+\de,1-\de)\) be \st for \(j=1,2\), \(\Om_j=\{h_j > \al_0\}\in \cm.\) Then, for all \(\ve>0\),
	\[\cH^{n+1}\left(\Om_1 \setminus\Om_2\right)\leq \ve C_{\de}^{-1}E_{\ve}(h_2)+(\al_0+1-\de)^2\norm{h_1-h_2}^2_{H^1(M)}.\]
	As a consequence, for \(j=1,2\), if \(\la_j\in  (-1+\de,1-\de)\) such that \(T_j=\del \db{\{h_j>\la_j\}}\in \cZ_n(M^{n+1};\bbz_2)\), 
	\[\cF(T_1,T_2)\leq 2\ve C_{\de}^{-1}(E_{\ve}(h_1)+E_{\ve}(h_2))+2(\al_0+1-\de)^2\norm{h_1-h_2}^2_{H^1(M)}\; \forall\ve>0.\]
	\label{lem 4.3}
\end{lem}\vfo
	We recall that \(X\) is a subcomplex of \(\sci^N[1]\) for some \(N \in \bbn\). There exists \(\al \in (-1+b, 1-b)\) \st for all \(i,j\in \bbn\) and \(x \in \pi^{-1}(X\cap \mathbb{Q}^N)\), \(\{h^i_j(x)> \al\}\in \cm\). Let us fix \(i\geq i_0\) (\(i_0\) is as in Proposition \ref{pro 4.1}); let \(j_0\in \bbn\) \st 
\[d_{H^1(M)}\left(u_i,h^i_{j_0}(\xt)\right)\leq \ve_i/2.\]
For simplicity, let us denote the map \(h^i_{j_0}\) by \(h\). We choose \(l_i\in \bbn\) \st if \(x,x'\) belong to a common cell in \(\xt[3^{-l_{i}}]\), \(\norm{h(x)-h(x')}_{H^1(M)}\leq \ve_i/2.\) Let \(x_i^*\in \xt[3^{-l_{i}}]_0\) \st 
\[d_{H^1(M)}\left(u_i,h(x_i^*)\right)=d_{H^1(M)}\left( u_i ,h\left(\xt[3^{-l_{i}}]_0\right) \right)\]
which is bounded by \(\ve_i\) by our choice of \(j_0\) and \(l_{i}\). Following the argument in \cite{Guaraco,GG1}, there exists a function \(\la:\tx \ra (-1+b, 1+b)\) \st for all \(x \in \tx\) the following conditions are satisfied.
\begin{itemize}\vth
	\item $\la(T(x))=-\la(x)$;
	\item \(\{h(x)>\la(x)\}, \{h(x)<\la(x)\}\in \cm\) with \(\db{\{h(x)>\la(x)\}}+\db{\{h(x)<\la(x)\}}=\db{M}\);
	\item Denoting \(\tilde{h}=F\circ h\) and \(\tilde{\la}=F\circ \la\); \(2\si_b\bM\big(\del\db{\{\tilde{h}(x)>\tilde{\la}(x)\}}\big) \leq E_{\ve_i}(h(x)). \)
\end{itemize}\vth
One can define a discrete, \(\bbz_2\)-equivariant map \(\tilde{\vp}_i:\xt[3^{-l_{i}}]_0 \ra \mathbf{I}_{n+1}(M^{n+1};\bbz_2)\) which is fine in the flat norm as follows (\(w_i\), \(t_i\) are as in Proposition \ref{pro 4.1}).
\begin{equation*}
	\tilde{\vp}_i(x)=
	\begin{cases}
	\db{\{\tilde{h}(x)>\tilde{\la}(x)\}} & \text{if } x \notin \{x_i^*,T(x_i^*)\};\\
	\db{\{w_i>t_i\}} & \text{if } x = x_i^*;\\
	\db{\{w_i<t_i\}} & \text{if } x = T(x_i^*).
	\end{cases}
\end{equation*}
If \(\vp_i=\del \circ \tilde{\vp}_i\), using Lemma \ref{lem 4.3}, fineness of \(\vp_i\) \wrt the flat norm
\[\mathbf{f}^{\cF}(\vp_i)\leq 4 \ve_i C^{-1}_b(L_{\ve_i}+\ve_i)+2(1-b+\al)^{-2}\ve_i^2\]
which converges to \(0\) as \(i \ra \infty\). Moreover, \(\bF(V,\md{\{w_i=t_i\}})\leq s\) implies that \(\bM\left(\db{\{w_i=t_i\}}\right)\leq L+s\). Hence,
\[\sup_{x\in X[3^{-l_{i}}]_0}\bM(\vp_i(x))\leq \max \left\{\frac{1}{2\si_b}(L_{\ve_i}+\ve_i),L+s\right\}.\]  
As argued in \cite{Guaraco,GG1}, one can apply the interpolation theorem of Zhou \cite{Zhou}*{Proposition 5.8} to produce a sequence of discrete maps whose fineness \wrt the mass norm converges to \(0\) and then, using the interpolation theorem of Marques and Neves \cite{MN_Willmore}*{Theorem 14.1}, one can find a sequence of maps continuous in the mass norm. More precisely, there exists \(i_0^*\geq i_0\) \st for all \(i\geq i_0^*\) there exist \(\Ph_i:X \ra \cZ_n(M^{n+1};\bM;\bbz_2)\) and \(\de_i > 0\) \st \(\Ph_i \in \Pi\), \(\de_i \ra 0\) and
\begin{equation*}
\sup_{x \in X}\bM\left(\Ph_i(x)\right)\leq \max \left\{\frac{1}{2\si_b}(L_{\ve_i}+\ve_i),L+s\right\}+\de_i. 
\end{equation*}
Moreover, \(\Ph_i(x)=\vp_i(x)\) for all $x\in X[3^{-l_{i}}]_0 $. In particular, \(\Ph_i(x_i^{*})=\vp_i(x_i^*)=\db{\{w_i=t_i\}}\); hence,
\[\bF\left(V,\md{\Ph_i(x_i^{*})}\right)\leq s.\]
\end{proof}
\begin{proof}[Proof of Theorem \ref{thm critical set}]
	By letting \(s\ra 0\), \(b \ra 0\) and \(i \ra \infty\) in the above Proposition \ref{pro 4.2}, we obtain Theorem \ref{thm critical set}. More precisely, let \(\{s_m\}_{m=1}^{\infty}\) be a sequence \st \(s_m \ra 0\). We choose \(b_m\in (0, b_0(s_m)]\) \st \(b_m \ra 0\). By Proposition \ref{pro 4.2}, for every \(m \in \bbn\), there exist \(\Ps_m:X \ra \cZ_n(M^{n+1};\bM;\bbz_2)\), \(i(m)\in \bbn\) and \(x_m^{\bullet}\in X\) \st \(\Ps_m \in \Pi\), \(i(m)>i(m-1)\),
	\begin{equation*}
	\sup_{x \in X}\bM\left(\Ps_m(x)\right)\leq \max \left\{\frac{1}{2\si_{b_m}}(L_{\ve_{i(m)}}+\ve_{i(m)}),L+s_m\right\} + s_m \; \text{ and } \; \bF\left(V,\md{\Ps_m(x_m^{\bullet})}\right)\leq s_m. \label{eqn of prop 4.2 bis}
	\end{equation*}
This implies \(\{\Ps_m\}_{m=1}^{\infty}\) is a minimizing sequence in \(\Pi\) and \(V \in \mathbf{C}\left(\{\Ps_m\}\right)\).	
\end{proof}
\nocite{*}
\medskip

\bibliographystyle{amsalpha}
\bibliography{apequalsac}

\end{document}